\documentclass[10pt,oneside,reqno]{amsart}

\makeatletter
\@namedef{subjclassname@2020}{\textup{2020} Mathematics Subject Classification}
\makeatother

\usepackage[top=2.5 cm,bottom=2 cm,left=2 cm,right=3 cm]{geometry}
\usepackage[utf8]{inputenc}
\usepackage{color}
\usepackage{amssymb}

\usepackage[english]{babel}
\usepackage{enumitem}

\usepackage{todonotes}

\usepackage[backref=page]{hyperref}

\usepackage{esint}
\usepackage{graphicx}
\usepackage{hyperref}
\usepackage{bm}
\usepackage{xcolor}

\newcommand*{\mailto}[1]{\href{mailto:#1}{\nolinkurl{#1}}}

\numberwithin{equation}{section}

\newtheorem{example}{Example}[section]

\newtheorem{theorem}[example]{Theorem}
 \newtheorem{proposition}{Proposition}
\newtheorem{lemma}{Lemma}
 \newtheorem{corollary}[example]{Corollary}
\newtheorem{remark}[example]{Remark}
\newtheorem*{maintheorem*}{Main Theorem}
\allowdisplaybreaks
\numberwithin{equation}{section}

\renewcommand{\i}{\ifmmode\mathit{\mathchar"7010 }\else\char"10 \fi}
\renewcommand{\j}{\ifmmode\mathit{\mathchar"7011 }\else\char"11 \fi}
\newcommand{\R}{\mathbb{R}}
\newcommand{\C}{\mathbb{C}}

\newcommand{\Z}{\mathbb{Z}}

\newcommand{\px}{\partial_x}

\newcommand{\I}{\mathrm{i}}

{%

\begin{enumerate}}%
{\end{enumerate}}

{%

\begin{enumerate}}%
{\end{enumerate}}

\begin{document}
\sloppy

\title[Nonlocal mKdV equation with step-like oscillating background]
{Riemann-Hilbert approach for the nonlocal modified Korteweg-de Vries equation with a step-like oscillating background}

\author[Rybalko]{Yan Rybalko}

\address[Yan Rybalko]{\newline
	Department of Mathematics, \newline University of Oslo, \newline
	PO Box 1053, Blindern -- 0316 Oslo, Norway
	\medskip
	\newline B. Verkin Institute for Low Temperature Physics and Engineering
	of the National Academy of Sciences of Ukraine,
	\newline 47 Nauky Ave., Kharkiv, 61103, Ukraine}
\email[]{rybalkoyan@gmail.com}



\thanks{The work of Yan Rybalko was supported 
by
the European Union’s Horizon Europe research and innovation programme under the Marie Sk\l{}odowska-Curie grant agreement No 101058830.}

\dedicatory{To the memory of Volodymyr Kotlyarov}
\date{\today}

\begin{abstract}
This work focuses on the Cauchy problem for the nonlocal modified Korteweg-de Vries equation
$$
u_t(x,t)+6u(x,t)u(-x,-t)u_x(x,t)+u_{xxx}(x,t)=0,
$$
with the oscillating step-like boundary conditions: 
$u(x,t)\to 0$ as $x\to-\infty$ and $u(x,t)\backsimeq A\cos(2Bx+8B^3t)$ as
$x\to\infty$, where $A,B>0$ are arbitrary constants.
The main goal is to develop the Riemann-Hilbert formalism for this problem,
paying a particular attention to the case of the ``pure oscillating step'' initial data, that is $u(x,0)=0$ for $x<0$ and $u(x,0)=A\cos(2Bx)$ for $x\geq0$.
Also, we derive three new families of two-soliton solutions, which correspond to the values of $A$ and $B$ satisfying $B<\frac{A}{4}$, $B>\frac{A}{4}$, and $B=\frac{A}{4}$.
\end{abstract}

\maketitle

\numberwithin{lemma}{section}
\numberwithin{proposition}{section}

\section{Introduction}

We study the
Cauchy problem for the nonlocal modified Korteweg-de Vries (nmKdV) equation:
\begin{subequations}
\label{nmkdvs}
\begin{align}
\label{nmkdvs-a}
&u_t(x,t)+6u(x,t)u(-x,-t)u_x(x,t)+u_{xxx}(x,t)=0,\quad x,t\in\R,\\
\label{nmkdvs-b}
&u(x,0)=u_0(x),
\end{align}
\end{subequations}
where the solution $u(x,t)$ satisfies the oscillating step-like boundary conditions, which read as follows:
\begin{equation}\label{bcs}
	u(x,t)=o(1),\quad x\to-\infty,\quad 
	u(x,t)=A\cos\left(2Bx+8B^3t\right)+o(1),\quad x\to\infty,
\end{equation}
for all fixed $t\in\R$ and with arbitrary $A,B>0$.
In particular, we consider the initial data $u_0(x)$ in the form 
of the ``pure oscillating step'' function, that is
\begin{equation}\label{sidR}
	u_0(x)=\begin{cases}
		0,&x<0,\\
		A\cos2Bx,&x\geq 0.
	\end{cases}
\end{equation}
Notice that the problem \eqref{nmkdvs}--\eqref{bcs} with
$B=0$ was previously covered in \cite{LSTZ25, LZ21, XF23}.

The nmKdV equation was derived by Ablowitz and Musslimani \cite{AM16} as a nonlocal reduction of a member of the Ablowitz-Kaup-Newell-Segur
(AKNS) hierarchy \cite{AKNS74}.
In more detail, consider the following integrable two-component AKNS system:
\begin{equation}\label{tcmkdv}
\begin{split}
	&u_t(x,t)+u_{xxx}(x,t)+6u(x,t)v(x,t)u_x(x,t)=0,\\
	&v_t(x,t)+v_{xxx}(x,t)+6u(x,t)v(x,t)v_x(x,t)=0.
\end{split}
\end{equation}
Then, taking $v(x,t)=u(-x,-t)$, one arrives at the nonlocal (two-place) mKdV equation \eqref{nmkdvs-a}, while for $v(x,t)=u(x,t)$ one has the conventional focusing mKdV equation, which reads
\begin{equation}\label{mKdV}
	u_t(x,t)+6u^2(x,t)u_x(x,t)+u_{xxx}(x,t)=0,\quad x,t\in\R.
\end{equation}
Apart from being integrable, the nmKdV equation has a parity-time (PT) symmetry property: if $u(x,t)$ satisfies \eqref{nmkdvs-a}, so does $u(-x,-t)$.
Then, \eqref{nmkdvs-a} admits various exact soliton solutions with zero and nonzero background, which were studied by the inverse scattering transform, Hirota method, and Darboux transformation in
\cite{AMO24, GP22, GP19, JZ17, LZ21, M21, ZTYZ24, ZY20, ZC23}.
Also, the nmKdV equation is an example of the two-place (Alice-Bob) systems, which govern phenomena having strong correlations 
of the events at non-neighboring places \cite{LH17}
(see also \cite{AM17, F16} for another examples of the nonlocal integrable models).
Notice that \cite[Equation (29)]{TLH18} suggests different two-place versions of the mKdV equation, derived from a two-vortex interaction model relevant for the atmospheric and oceanic dynamics.
Finally, we observe that the Cauchy problem for the general two-component system \eqref{tcmkdv}
with the initial data satisfying $v_0(x)=u_0(-x)$ reduces to the Cauchy problem for the nmKdV equation, provided uniqueness of the solution holds.

The step-like problems for the integrable equations have been studied since the work by Gurevich and Pitaevskii \cite{GP73}, who considered the Korteweg-de Vries (KdV) equation with the step-like boundary conditions
\begin{equation}\label{bcg}
	u(x,t)=o(1),\quad x\to-\infty,\quad
	u(x,t)=f(x,t)+o(1),\quad x\to\infty,
\end{equation}
with $f(x,t)$ being a nonzero constant function
(see, for example, \cite{EPT24, EGKT13, EHS17} and references therein for the recent advances in the study of this problem).
The mKdV equation \eqref{mKdV} with the step-like initial data was considered for the first time in \cite{KhK89} (see also \cite{B95}).
The Riemann-Hilbert approach for these problems was developed in \cite{BM19, GM20, KM10, KM15}, where, in particular, the authors rigorously studied the long-time asymptotic behavior of the solution in different space-time regions.
Considering the Cauchy problem for the mKdV equation with the boundary conditions \eqref{bcg}, the function $f(x,t)$ typically satisfies the 
\textit{nonlinear} equation \eqref{mKdV}.
We refer the reader to \cite{CAA16, CP19, LS23}, which deal with the periodic solutions of the focusing mKdV equation.
In contrast, since the nonlinear term in the nonlocal mKdV equation involves values of the solution at both $x$ and $-x$,
the function $f(x,t)$ on the background
\eqref{bcg} satisfies the \textit{linear} equation
\begin{equation*}
	f_t(x,t)+f_{xxx}(x,t)=0,
\end{equation*}
which has a simple non-constant periodic solution
$f(x,t)=A\cos\left(2Bx+8B^3t\right)$, as per \eqref{bcs}.
Notice that the complex focusing mKdV equation, where the nonlinear term has the form $6|u|^2u_x$ and $u$ is a complex-valued function, admits a plane wave solution in the form $Ae^{-2\I Bx+2\I\omega t}$, with $\omega=6A^2B-4B^3$ \cite{TWZ25, WXF23}.

The work \cite{RST25} deals with the oscillating step-like problem for another
two-place integrable equation, namely, the nonlocal nonlinear Schr\"odinger (NNLS) equation, which reads as follows \cite{AM13}:
\begin{equation}
	\label{NNLS}
	\mathrm{i}q_t(x,t)+q_{xx}(x,t)
	+2q^{2}(x,t)\bar{q}(-x,t)=0,
	\quad q\in\mathbb{C},\quad
	\mathrm{i}^2=-1.
\end{equation}
Here and below, $\bar{q}$ denotes the complex conjugate of $q$.
In \cite{RST25}, the authors take
$f(x,t)=Ae^{2\I Bx-4\I B^2t}$ in \eqref{bcg} (with $q$ instead of $u$), and
develop the inverse scattering transform method for this problem.
The work \cite{RST25} shows that the properties of zeros
of the spectral function $a_1(k)$ for $\mathrm{Im}\,k\geq0$, associated to the ``pure step'' initial data
\begin{equation*}
	q_{0}(x)=\begin{cases}
		0,&x\leq0,\\
		A e^{2\I Bx},&x>0,
	\end{cases}
\end{equation*}
depend on the sign of $4B^2-A^2$.
In more detail, if $4B^2<A^2$, the function $a_1(k)$ has one simple purely imaginary zero; for $4B^2>A^2$, $a_1(k)$ has no zeros, and $a_1(k)$ has one zero of multiplicity two at $k=0$ for $4B^2=A^2$, see \cite[Proposition 2.4]{RST25}.
Notably, the associated basic Riemann-Hilbert problem has a singularity 
on the real axis at $k=-B$,
which leads to the qualitatively different large-time asymptotic behavior of the solution for $B$ having different sign \cite[Theorems 3.4, 3.5]{RST25}.
We also refer the reader to \cite{BFP16, BLS22, DPMV14, FLQ25} and references therein for a related discussion of the step-like problems for the conventional NLS equation.

In this paper, we develop the inverse scattering transform method in the form of the Riemann-Hilbert problem for the Cauchy problem \eqref{nmkdvs}--\eqref{bcs}.
We show that the spectral functions associated to the 
step-like initial data \eqref{sidR} have different properties 
depending on the sign of $B-\frac{A}{4}$: the behavior of zeros in the upper complex half-pane is different for $B<\frac{A}{4}$, $B>\frac{A}{4}$,
and $B=\frac{A}{4}$, cf.~\cite{RST25}.
Notably, we prove that
the scattering data,
particularly zeros of the spectral functions $a_j(k)$, $j=1,2$,
associated to the small (in the $L^1$ norm) perturbations of the pure step initial profile are fully determined by one spectral function $b(k)$, see Propositions \ref{za1a2} and \ref{tza1}.
Also, we establish that the associated Riemann-Hilbert problem has \textit{two} singular points on the real axis at $k=\pm B$,
not one as for the NNLS equation (see Proposition \ref{Bcon}).

Considering the associated Riemann-Hilbert problem in the reflectionless case
(i.e., with $b(k)=0$),
we obtain three new families of two-soliton solutions 
of the nmKdV equation satisfying boundary conditions \eqref{bcs},
see \eqref{usI}, \eqref{usII}, and \eqref{usIII}.
Here, each family of solutions corresponds to two simple real zeros
at $k=\pm B$, and either two simple purely imaginary zeros at
$k=\I k_j$, $0<k_1<k_2$,
two simple zeros at $k=p_1$ and $k=-\bar{p}_1$, $\mathrm{Re}\,p_1<0$,
$\mathrm{Im}\,p_1>0$, or
one purely imaginary zero of multiplicity two at $k=\I\ell_1$, $\ell_1>0$, in the upper complex half-plane.
Taking into account that the spectral data are defined in terms of $b(k)$ only, the values of $k_j$, $j=1,2$, $p_1$, and $\ell_1$ are fully determined by $A$ and $B$ for each soliton solution, see \eqref{zrlc}.
Another remarkable feature of the obtained solitons is the blow-up in the uniform norm along certain curves in the $(x,t)$-plane, which can be recognized from Figures \ref{fusI}, \ref{fusII}, and \ref{fusIII}.

The article is organized as follows.
Section \ref{istrh} is devoted to the Riemann-Hilbert formalism for the 
problem \eqref{nmkdvs} 
with the oscillating step-like boundary conditions \eqref{bcs}.
More precisely, in Sections \ref{Ef} and \ref{Sd}, we introduce the Jost solutions and define the associated spectral functions.
Section \ref{ps-id} studies the pure step initial data \eqref{sidR},
while Section \ref{Asssd} describes assumptions on the scattering data, associated to the general initial data $u_0(x)$, and characterizes them in terms of one spectral function $b(k)$.
Then, we construct the basic Riemann-Hilbert problem in Section \ref{BRHp}, which is formulated in Theorem \ref{TBRHp}.
Finally, Section \ref{twsol} defines three new families of two-soliton solutions, as detailed in Theorem \ref{Thtws}.


\medskip
\textbf{Notations.}
$X_{ij}$ denotes the $(i,j)$ element of the matrix $X$,
$X^{(i)}$ stands for the $i$-th column of $X$,
and $X^T$ is the matrix transpose to $X$.
We denote the $2\times2$ identity matrix by $I$, and
use the following Pauli matrices:
\begin{equation}\label{P-mat}
	\sigma_1=\begin{pmatrix}0& 1\\1 & 0\end{pmatrix},\quad
	\sigma_3=\begin{pmatrix}1& 0\\0 & -1\end{pmatrix}.
\end{equation}
Also, we use notations
$\mathbb{C}^{\pm}=\left\{k\in\mathbb{C}\,|\pm
\mathrm{Im}\,k>0\right\}$ and
$\overline{\mathbb{C}^{\pm}}=\left\{k\in\mathbb{C}\,|\pm
\mathrm{Im}\,k\geq0\right\}$.

\section{Inverse scattering transform and the Riemann-Hilbert problem}
\label{istrh}
In this section, we develop the inverse scattering transform method in the form of the Riemann-Hilbert problem for the Cauchy problem \eqref{nmkdvs} with the boundary conditions \eqref{bcs}.
To this end, we define the Jost solutions $\Psi_j(x,t,k)$, $j=1,2$, of the linear Volterra integral equations (see \eqref{Psi1}--\eqref{Psi2} below),
which are closely related to the solutions $\Phi_j(x,t,k)$ of the Lax pair for the nmKdV equation, see \eqref{Phi}.
Observe that one of the distinctive features of the $2\times2$ matrix solutions $\Psi_j$, $j=1,2$, is that their columns have singularities at $k=\pm B$, described in
item (vi) of Proposition \ref{pPsi}.
Then, we introduce the spectral functions $a_j(k)$, $j=1,2$, and $b(k)$ in terms of the known initial data $u_0(x)$ through the determinants of the columns of $\Psi_j(0,0,k)$, see \eqref{sd}.

The crucial properties of the spectral functions $a_1(k)$ and $a_2(k)$ are their zeros in the closed upper and lower complex half-planes, respectively.
To motivate our assumptions on zeros of $a_j(k)$, $j=1,2$,
corresponding to the general initial value $u_0(x)$,
we investigate zeros of these functions for the pure oscillating step function \eqref{sidR}, as discussed in Section \ref{ps-id}.
We prove that depending on the values of $A$ and $B$, there are three possible scenarios of behavior of zeros of $a_1(k)$,
see Proposition \ref{a_1s}, which highlights the relevance of considering the scattering data described in
Cases I, II, and III, see Section \ref{Asssd}.
Moreover, we cover three more classes of spectral functions,
see Cases $\widetilde{\mathrm{I}}$, $\widetilde{\mathrm{II}}$, and $\widetilde{\mathrm{III}}$
in Sections \ref{Asssd},
where $a_2(k)$ has two simple zeros at $k=\pm B$.
Also, we establish that both $a_1(k)$ and $a_2(k)$ are fully determined by the spectral function $b(k)$; in particular, we obtain the exact formulas for their zeros in terms of $b(k)$, as detailed in Propositions \ref{za1a2} and \ref{tza1}.

Finally, defining the sectionally meromorphic $2\times 2$ matrix $M(x,t,k)$
in terms of 
the columns of $\Psi_j(x,t,k)$ and $a_j(k)$, $j=1,2$,
we show that this matrix satisfies a Riemann-Hilbert problem with a jump condition across $\R\setminus\{B,-B\}$.
Also, it has residue conditions at the zeros of $a_1(k)$,
and certain singularity conditions at $k=\pm B$. 
Notably, the solution $u(x,t)$ of the initial Cauchy problem \eqref{nmkdvs} can be recovered by the inverse transform in terms of $(1,2)$ or $(2,1)$ elements of $M(x,t,k)$ (see Section \ref{BRHp} and Theorem \ref{TBRHp} for details).
\subsection{Eigenfunctions}\label{Ef}

Consider the following system of linear equations (recall \eqref{P-mat}):
\begin{equation}
\label{LPnmkdv}
\begin{split}
	&\Phi_{x}+\I k\sigma_{3}\Phi=U(x,t)\Phi,\\
	&\Phi_{t}+4\I k^{3}\sigma_{3}\Phi=V(x,t,k)\Phi,
\end{split}
\end{equation}
where $\Phi=\Phi(x,t,k)$ is a $2\times2$ matrix-valued function, $k\in\mathbb{C}$ is an auxiliary (spectral) parameter, 
and the coefficients $U(x,t)$ and $V(x,t,k)$ are given in terms of the solution $u(x,t)$ of \eqref{nmkdvs-a} as follows:
\begin{equation}
	\label{U}
	U(x,t)=\begin{pmatrix}
		0& u(x,t)\\
		-u(-x,-t)& 0\\
	\end{pmatrix},
\end{equation}
\begin{equation}
	\label{V}
	V(x,t,k)=\begin{pmatrix}
		\tilde{A}& \tilde{B}\\
		\tilde{C}& -\tilde{A}\\
	\end{pmatrix},
\end{equation}
with 
\begin{equation*}
\begin{split}
	&\tilde{A}=2\I k u(x,t)u(-x,-t)-u_x(x,t)u(-x,-t)
	+u(x,t)(u(-x,-t))_x,\\
	&\tilde{B}=4k^2u(x,t)+2\I ku_x(x,t)
	-2u^2(x,t)u(-x,-t)-u_{xx}(x,t),\\
	&\tilde{C}=-4k^2u(-x,-t)+2\I k(u(-x,-t))_x
	+2u(x,t)u^2(-x,-t)+(u(-x,-t))_{xx}.
\end{split}
\end{equation*}
Then system \eqref{LPnmkdv} is the Lax pair for the nmKdV equation:
its compatibility condition
\begin{equation}\label{comp}
	U_t-V_x
	+[U-\I k\sigma_3,V-4\I k^3\sigma_3]=0,
\end{equation}
where $[X,Y]=XY-YX$ is the matrix commutator,
is equivalent to \eqref{nmkdvs-a}.

Since the solution $u(x,t)$ satisfies boundary conditions \eqref{bcs}, we have the following limits of $U$ and $V$ as $x\to\pm\infty$ (here $t$ and $k$ are fixed):
\begin{equation*}
	U(x,t)-U_{\pm}(x,t)=o(1),\,\,
	x\to\pm\infty,\quad
	V(x,t,k)-V_{\pm}(x,t,k)=o(1),\,\,
	x\to\pm\infty,
\end{equation*}
where
\begin{equation}
	\label{Upm}
	U_+(x,t)=
	\begin{pmatrix}
		0 & A\cos(2Bx+8B^3t)\\
		0 & 0
	\end{pmatrix},
	\qquad\quad\,\,\,\,\,
	U_-(x,t)=
	\begin{pmatrix}
		0 & 0\\
		-A\cos(2Bx+8B^3t) & 0
	\end{pmatrix},
\end{equation}
and
\begin{equation}\label{Vpm}
\begin{split}
	&V_+(x,t,k)=
	\begin{pmatrix}
		0 & 4A(k^2+B^2)\cos(2Bx+8B^3t)
		-4\I ABk\sin(2Bx+8B^3t)\\
		0 & 0
	\end{pmatrix},\\
	&V_-(x,t,k)=
	\begin{pmatrix}
		0 & 0\\
		-4A(k^2+B^2)\cos(2Bx+8B^3t)-4\I ABk\sin(2Bx+8B^3t)& 0
	\end{pmatrix}.
\end{split}
\end{equation}

Now we introduce the ``background solutions'' $\Phi_\pm(x,t,k)$ solving the Lax pair \eqref{LPnmkdv} with $U_\pm$ and $V_\pm$ instead of $U$ and $V$, respectively.
Direct computations show that $\Phi_\pm(x,t,k)$ can be taken as follows:
\begin{equation}\label{phi+-}
	\Phi_\pm(x,t,k)=
	N_\pm(x,t,k)e^{-(\I kx+4\I k^3t)\sigma_3},
\end{equation}
where
\begin{equation}\label{N+-}
\begin{split}
	&N_+(x,t,k)=
	\begin{pmatrix}
		1&-\frac{A}{2\left(k^2-B^2\right)}
		\left(
		B\sin\left(2Bx+8B^3t\right)+\I k\cos\left(2Bx+8B^3t\right)		
		\right)\\
		0&1
	\end{pmatrix},\\
	&N_-(x,t,k)=
	\begin{pmatrix}
		1&0\\
		\frac{A}{2\left(k^2-B^2\right)}
		\left(
		B\sin\left(2Bx+8B^3t\right)-\I k\cos\left(2Bx+8B^3t\right)		
		\right)&1
	\end{pmatrix}.
\end{split}
\end{equation}
Observe that \eqref{phi+-} reduces to \cite[Equation (17)]{XF23} for $B=0$.

Using the background solutions $\Phi_\pm$, we formally define the $2\times2$ matrix-valued functions $\Psi_j(x,t,k)$, $j=1,2$, as the solutions of the following
linear Volterra integral equations:
\begin{equation}
	\label{Psi1}
	\begin{split}
		\Psi_1(x,t,k)=N_-(x,t,k)+
		\int_{-\infty}^{x}G_-(x,y,t,k)
		\left(U-U_-\right)(y,t)
		\Psi_1(y,t,k)e^{\I k(x-y)\sigma_3}\,dy,
	\end{split}
\end{equation}
\begin{equation}
	\label{Psi2}
	\begin{split}
		\Psi_2(x,t,k)=N_+(x,t,k)-
		\int^{\infty}_{x}G_+(x,y,t,k)
		\left(U-U_+\right)(y,t)
		\Psi_2(y,t,k)e^{\I k(x-y)\sigma_3}\,dy,
	\end{split}
\end{equation}
where $N_\pm$ and $U_\pm$ are given in \eqref{N+-} and \eqref{Upm}, respectively, and
(see \eqref{phi+-})
\begin{equation}
	\label{G+-}
	G_\pm(x,y,t,k)=
	\Phi_\pm(x,t,k)
	\Phi_\pm^{-1}(y,t,k)
	=N_\pm(x,t,k)e^{\I k(y-x)\sigma_3}N_\pm^{-1}(y,t,k).
\end{equation}
In the next proposition, we relate $\Psi_j$, $j=1,2$, to the solutions of the Lax pair \eqref{LPnmkdv}.
Also, we summarize their important symmetry and analytical properties, thus generalizing \cite[Propositon 1]{XF23}, which considers the case $B=0$.

\begin{proposition}
	\label{pPsi}
	Assume that $xu(x,t)\in L^1(-\infty,a)$ and
	$\left(
	u(x,t)-A\cos(2Bx+8B^3t)
	\right)\in L^1(a,\infty)$ with respect to the spatial variable $x$, for all fixed $t,a\in\R$.
	Then the matrices $\Psi_j(x,t,k)$, $j=1,2$, given in \eqref{Psi1}--\eqref{Psi2}, have the following properties:
	\begin{enumerate}[label=(\roman*)]
		
		\item the functions $\Phi_j(x,t,k)$, $j=1,2$ defined by
		\begin{equation}
			\label{Phi}
			\Phi_j(x,t,k)=\Psi_j(x,t,k)
			e^{-(\I kx+4\I k^3t)\sigma_3},\quad j=1,2,
			\quad k\in\R\setminus \{-B,B\},
		\end{equation}
		are the Jost solutions of the Lax pair equations (\ref{LPnmkdv}) satisfying 
		the boundary conditions, see \eqref{phi+-}
		\begin{equation*} 
		\begin{split}
			&\Phi_1(x,t,k)-\Phi_-(x,t,k)=o(1),\quad x\to-\infty,\\
			&\Phi_2(x,t,k)-\Phi_+(x,t,k)=o(1),\quad x\to\infty;
		\end{split}
		\end{equation*}

		\item the columns $\Psi_1^{(1)}(x,t,k)$ and $\Psi_2^{(2)}(x,t,k)$ are  
		analytic in $k\in\C^+$ and continuous in 
		$\overline{\C^+}\setminus\{-B,B\}$.
		Moreover,
		$$
		\Psi_1^{(1)}(x,t,k)=
		\begin{pmatrix}
			1\\
			0\end{pmatrix}
		+O\left(k^{-1}\right),\quad \Psi_2^{(2)}(x,t,k)=
		\begin{pmatrix}
			0\\
			1\end{pmatrix}
		+O\left(k^{-1}\right), \quad 
		k\to\infty, \quad  k\in\C^+;
		$$
		\item the columns $\Psi_1^{(2)}(x,t,k)$ and $\Psi_2^{(1)}(x,t,k)$ are  
		analytic in $k\in\C^-$ and continuous in $\overline{\C^-}$.
		Moreover,
		$$
		\Psi_1^{(2)}(x,t,k)=
		\begin{pmatrix}
			0\\
			1\end{pmatrix}
		+O\left(k^{-1}\right),\quad
		\Psi_2^{(1)}(x,t,k)=
		\begin{pmatrix}
			1\\
			0\end{pmatrix}
		+O\left(k^{-1}\right),\quad k\to\infty,\quad
		k\in\C^-;
		$$
		
		\item $\det\Psi_j(x,t,k)=1$, $j=1,2$, for any $x,t\in\R$ and $k\in\R\setminus\{-B,B\}$;
		
		\item $\Psi_j(x,t,k)$, $j=1,2$, satisfy the following  symmetry relations (recall \eqref{P-mat}):
		\begin{equation}\label{Psi-sym}
		\begin{split}
			&\sigma_1\Psi_1(-x,-t,k)\sigma_1^{-1}
			=\Psi_2(x,t,k),\\
			&\sigma_1\overline{\Psi_1}\left(-x,-t,-\bar{k}\right)\sigma_1^{-1}
			=\Psi_2(x,t,k),
		\end{split}
		\end{equation}
		which hold for the first and the second columns for
		$k\in\overline{\C^-}$ and 
		$k\in\overline{\C^+}\setminus\{-B,B\}$, respectively;

		\item the columns of $\Psi_j$, $j=1,2$, have the following expansions as $k\to\pm B$,
		
		\begin{equation*}
		\begin{split}
			&\Psi_1^{(1)}(x,t,k)=\frac{1}{k-B}
			\begin{pmatrix}v_1(x,t)\\v_2(x,t)\end{pmatrix}+O(1),
			\quad k\to B,\\
			&\Psi_1^{(2)}(x,t,k)=\frac{4\I}{A}
			e^{-2\I Bx-8\I B^3t}
			\begin{pmatrix}v_1(x,t)\\v_2(x,t)\end{pmatrix}+O(k-B),\quad
			k\to B,
		\end{split}
		\end{equation*}
		\begin{equation*}
		\begin{split}
			&\Psi_2^{(1)}(x,t,k)=\frac{4\I}{A}
			e^{2\I Bx+8\I B^3t}
			\begin{pmatrix}v_2(-x,-t)\\v_1(-x,-t)\end{pmatrix}+O(k-B),
			\quad k\to B,\\
			&\Psi_2^{(2)}(x,t,k)=\frac{1}{k-B}
			\begin{pmatrix}v_2(-x,-t)\\v_1(-x,-t)\end{pmatrix}+O(1),
			\quad k\to B,
		\end{split}
		\end{equation*}
		\begin{equation*}
		\begin{split}
			&\Psi_1^{(1)}(x,t,k)=-\frac{1}{k+B}
			\begin{pmatrix}\bar{v}_1(x,t)\\\bar{v}_2(x,t)\end{pmatrix}+O(1),
			\quad k\to -B,\\
			&\Psi_1^{(2)}(x,t,k)=-\frac{4\I}{A}
			e^{2\I Bx+8\I B^3t}
			\begin{pmatrix}\bar{v}_1(x,t)\\
				\bar{v}_2(x,t)\end{pmatrix}+O(k+B),\quad
			k\to -B,
		\end{split}
		\end{equation*}
		and
		\begin{equation*}
		\begin{split}
			&\Psi_2^{(1)}(x,t,k)=-\frac{4\I}{A}
			e^{-2\I Bx-8\I B^3t}
			\begin{pmatrix}\bar{v}_2(-x,-t)\\
				\bar{v}_1(-x,-t)\end{pmatrix}+O(k+B),
			\quad k\to -B,\\
			&\Psi_2^{(2)}(x,t,k)=-\frac{1}{k+B}
			\begin{pmatrix}\bar{v}_2(-x,-t)\\\bar{v}_1(-x,-t)\end{pmatrix}+O(1),\quad
			k\to -B,
		\end{split}
		\end{equation*}
		where $v_j(x,t)$, $j=1,2$, solve the following system of linear Volterra integral equations:
		\begin{equation}\label{v1v2-s}
			\begin{split}
				&v_1(x,t)=\int_{-\infty}^x u(y,t)v_2(y,t)\,dy,\\
				&v_2(x,t)=-\I\frac{A}{4}e^{2\I Bx+8\I B^3t}
				-\int_{-\infty}^xu(-y,-t)v_1(y,t)e^{2\I B(x-y)}\,dy.
			\end{split}
		\end{equation}

	\end{enumerate}
	
\end{proposition}
\begin{proof}
	Items (i)--(iii) follow directly from the compatibility condition \eqref{comp} and the definition of $\Psi_j(x,t,k)$ given in \eqref{Psi1}--\eqref{Psi2}.
	Taking into account that the matrices $(U-\I k\sigma_3)$ and
	$(V-4\I k^3\sigma_3)$ are traceless, we conclude item (iv) from the Liouville's theorem.
	Using the following symmetry relations:
	\begin{equation*}
		\begin{split}
			&\sigma_1 e^{a\sigma_3}\sigma_1^{-1}=e^{-a\sigma_3},
			\quad a\in\C,
			\quad\sigma_1 U(-x,-t)\sigma_1^{-1}=-U(x,t),\\
			&\quad\sigma_1N_-(-x,-t,k)\sigma_1^{-1}=N_+(x,t,k),\quad
			\sigma_1\overline{N_-}\left(-x,-t,-\bar{k}\right)
			\sigma_1^{-1}=N_+(x,t,k),
		\end{split}
	\end{equation*}
	we obtain item (v).
	\medskip
	
	To prove item (vi), we assume that $\Psi_1(x,t,k)$ has the following expansion as $k\to B$:
	\begin{equation}\label{Psi-1-B}
		\Psi_1^{(1)}(x,t,k)=\frac{1}{k-B}
		\begin{pmatrix}v_1(x,t)\\v_2(x,t)\end{pmatrix}+O(1),\quad
		\Psi_1^{(2)}(x,t,k)=
		\begin{pmatrix}w_1(x,t)\\w_2(x,t)\end{pmatrix}+O(k-B),\quad
		k\to B.
	\end{equation}
	Combining \eqref{Psi-1-B} and symmetry relations \eqref{Psi-sym}, we obtain the following behavior of the columns of $\Psi_2$, as
	$k\to\pm B$:
	\begin{equation*}
		\begin{split}
			&\Psi_2^{(1)}(x,t,k)=
			\begin{pmatrix}w_2(-x,-t)\\w_1(-x,-t)\end{pmatrix}+O(k-B),
			\quad k\to B,\\
			&\Psi_2^{(2)}(x,t,k)=\frac{1}{k-B}
			\begin{pmatrix}v_2(-x,-t)\\v_1(-x,-t)\end{pmatrix}+O(1),\quad
			k\to B,
		\end{split}
	\end{equation*}
	and
	\begin{equation}\label{Psi-2--B}
		\begin{split}
			&\Psi_2^{(1)}(x,t,k)=
			\begin{pmatrix}\overline{w}_2(-x,-t)\\
				\overline{w}_1(-x,-t)\end{pmatrix}+O(k+B),
			\quad k\to -B,\\
			&\Psi_2^{(2)}(x,t,k)=-\frac{1}{k+B}
			\begin{pmatrix}\bar{v}_2(-x,-t)\\
				\bar{v}_1(-x,-t)\end{pmatrix}+O(1),
			\quad k\to -B.
		\end{split}
	\end{equation}
	Using the first symmetry relation in \eqref{Psi-sym}, we conclude from \eqref{Psi-2--B} that
	\begin{equation}\label{Psi-1--B}
		\begin{split}
			&\Psi_1^{(1)}(x,t,k)=-\frac{1}{k+B}
			\begin{pmatrix}\bar{v}_1(x,t)\\\bar{v}_2(x,t)\end{pmatrix}+O(1),
			\quad k\to -B,\\
			&\Psi_1^{(2)}(x,t,k)=
			\begin{pmatrix}\overline{w}_1(x,t)\\
				\overline{w}_2(x,t)\end{pmatrix}+O(k+B),\quad
			k\to -B.
		\end{split}
	\end{equation}
	
	Equation \eqref{G+-} implies that $G_-$ has the following form:
	\begin{equation}\label{G-ex}
		G_-(x,y,t,k)=
		\begin{pmatrix}
			e^{-\I k(x-y)}&0\\
			\frac{A}{k^2-B^2}\tilde{g}(x,y,t,k)&
			e^{\I k(x-y)}
		\end{pmatrix},
	\end{equation}
	with (we drop the arguments of $\tilde{g}$ for simplicity)
	\begin{equation}\label{g-t}
		\begin{split}
			\tilde{g}=
			&\cos B\left(x+y+8B^2t\right)
			\left(B\sin B(x-y)\cos k(x-y)
			-k\sin k(x-y)\cos B(x-y)\right)\\
			&+\I\sin B\left(x+y+8B^2t\right)
			\left(k\sin B(x-y)\cos k(x-y)
			-B\sin k(x-y)\cos B(x-y)\right).
		\end{split}
	\end{equation}
	Using that
	\begin{equation*}
		\begin{split}
			B\sin B(x-y)\cos k(x-y)
			-k\sin k(x-y)\cos B(x-y)=
			&-(k-B)\left(\frac{1}{2}\sin2B(x-y)+B(x-y)\right)\\
			&+O\left((k-B)^2\right),\quad k\to B,
		\end{split}
	\end{equation*}
	and
	\begin{equation*}
		\begin{split}
			k\sin B(x-y)\cos k(x-y)
			-B\sin k(x-y)\cos B(x-y)=
			&(k-B)\left(\frac{1}{2}\sin2B(x-y)-B(x-y)\right)\\
			&+O\left((k- B)^2\right),\quad k\to B,
		\end{split}
	\end{equation*}
	we obtain from \eqref{G-ex} and \eqref{g-t} the following expansion of
	$G_-$ as $k\to B$:
	\begin{equation}\label{G-B}
		\begin{split}
			G_-(x,y,t,k)=
			&\begin{pmatrix}
				e^{-\I B(x-y)}&0\\
				-\frac{A}{2B}\left(\frac{1}{2}\sin2B(x-y)
				e^{-\I B(x+y+8B^2t)}+B(x-y)e^{\I B(x+y+8B^2t)}\right)&
				e^{\I B(x-y)}
			\end{pmatrix}\\
			&+O\left((k- B)^2\right),\quad k\to B.
		\end{split}
	\end{equation}
	Combining \eqref{Psi1}, \eqref{Psi-1-B}, and \eqref{G-B}, we obtain the following system of integral equations for $(v_1,v_2)(x,t)$:
	\begin{equation}\label{v1v2-1}
		\begin{split}
			&v_1(x,t)=\int_{-\infty}^x u(y,t)v_2(y,t)\,dy,\\
			&v_2(x,t)e^{-2\I Bx}=-\I\frac{A}{4}e^{8\I B^3t}
			+I_1(x,t),
		\end{split}
	\end{equation}
	with
	\begin{equation*}
		\begin{split}
			&I_1(x,t)=\int_{-\infty}^x\left(
			\left(A\cos(2By+8B^3t)-u(-y,-t)\right)
			v_1(y,t)e^{-2\I By}\right.\\
			&\qquad\qquad\qquad\quad\,
			\left.-\frac{A}{2B}u(y,t)\left(\frac{1}{2}\sin2B(x-y)
			e^{-2\I B(x+y+4B^2t)}+B(x-y)
			e^{8\I B^3t}\right)v_2(y,t)\right)dy.
		\end{split}
	\end{equation*}
	Equations \eqref{v1v2-1} imply that
	\begin{equation*}
		\begin{split}
			\px\left(v_2(x,t)e^{-2\I Bx}\right)&=
			\left(A\cos(2Bx+8B^3t)-u(-x,-t)\right)v_1(x,t)e^{-2\I Bx}\\
			&\quad\,-\frac{A}{2}\left(
			e^{-4\I B(x+2B^2t)}+e^{8\I B^3t}
			\right)v_1(x,t)\\
			&=-u(-x,-t)v_1(x,t)e^{-2\I Bx},
		\end{split}
	\end{equation*}
	which yields \eqref{v1v2-s}.
	
	Now we derive the system of equations for $(w_1,w_2)$ by using similar arguments as for $(v_1,v_2)$.
	In more detail, \eqref{Psi1}, \eqref{Psi-1-B}, and \eqref{G-B} imply that
	(cf.~\eqref{v1v2-1})
	\begin{equation}
		\label{w1w2-1}
		\begin{split}
			&w_1(x,t)=\int_{-\infty}^x u(y,t)w_2(y,t)e^{-2\I B(x-y)}\,dy,\\
			&w_2(x,t)=1+I_2(x,t),
		\end{split}
	\end{equation}
	where
	\begin{equation*}
		\begin{split}
			&I_2(x,t)=\int_{-\infty}^x\left(
			\left(A\cos(2By+8B^3t)-u(-y,-t)\right)
			w_1(y,t)\right.\\
			&\qquad\qquad\qquad\quad\,
			\left.-\frac{A}{2B}u(y,t)\left(\frac{1}{2}\sin2B(x-y)
			e^{-2\I B(x+4B^2t)}+B(x-y)
			e^{2\I B(y+4B^2t)}\right)w_2(y,t)\right)dy.
		\end{split}
	\end{equation*}
	Then, system \eqref{w1w2-1} yields the following differential equation for $w_2$:
	\begin{equation}\label{w2-x}
		\px w_2(x,t)=-u(-x,-t)w_1(x,t).
	\end{equation}
	We conclude from \eqref{w1w2-1} and \eqref{w2-x} that
	$(w_1,w_2)$ satisfies the system of integral equations, which reads
	\begin{equation}
		\label{w1w2-s}
		\begin{split}
			&w_1(x,t)=\int_{-\infty}^x u(y,t)w_2(y,t)e^{-2\I B(x-y)}\,dy,\\
			&w_2(x,t)=1-\int_{-\infty}^x u(-y,-t)w_2(y,t)\,dy.
		\end{split}
	\end{equation}
	Combining \eqref{v1v2-s} and \eqref{w1w2-s}, we obtain the following relation between $(v_1,v_2)$ and $(w_1,w_2)$:
	\begin{equation*}
		\begin{pmatrix}
			w_1(x,t)\\
			w_2(x,t)
		\end{pmatrix}
		=\frac{4\I}{A}e^{-2\I Bx-8\I B^3t}
		\begin{pmatrix}
			v_1(x,t)\\
			v_2(x,t)
		\end{pmatrix},
	\end{equation*}
	which, together with \eqref{Psi-1-B}--\eqref{Psi-1--B}, implies item (vi) for $B>0$.
\end{proof}

\subsection{Scattering data}
\label{Sd}
The solutions $\Phi_1(x,t,k)$ and $\Phi_2(x,t,k)$ of system \eqref{LPnmkdv}, see \eqref{Phi}, are related as follows:
\begin{equation}
	\label{S}
	\Phi_1(x,t,k)=\Phi_2(x,t,k)S(k),\quad 
	x,t\in\R,\quad
	k\in\R\setminus\{-B,B\},
\end{equation}
where $S(k)$ is the so-called scattering matrix.
Notice that item (iv) of Proposition \ref{pPsi} implies that
$\det S(k)=1$ for all $k$.
Using symmetry relations \eqref{Psi-sym}, we can write the scattering matrix $S(k)$ as follows:
\begin{equation}\label{S-m}
S(k)=
\begin{pmatrix}
	a_1(k)& -b(k)\\
	b(k)& a_2(k)
\end{pmatrix},\quad k\in\R\setminus\{-B,B\}.
\end{equation}
Combining \eqref{Phi}, \eqref{S}, \eqref{S-m}, and applying the Cramer's rule, we conclude that the spectral functions $a_j(k)$, $j=1,2$, and $b(k)$ can be found in terms of the following determinants:
\begin{subequations}\label{sd}
	\begin{align}
		\label{sda}
		&a_1(k)=\det\left(
		\Psi_1^{(1)}(0,0,k),\Psi_2^{(2)}(0,0,k)
		\right),
		\quad k\in\overline{\C^+}
		\setminus\{-B,B\},\\
		\label{sda2}
		&a_2(k)=\det\left(
		\Psi_2^{(1)}(0,0,k),\Psi_1^{(2)}(0,0,k)
		\right),
		\quad k\in\overline{\C^-},\\
		\label{sdb}
		&b(k)=\det\left(
		\Psi_2^{(1)}(0,0,k),\Psi_1^{(1)}(0,0,k)
		\right),
		\quad k\in\R\setminus\{-B,B\}.
	\end{align}
\end{subequations}
Notice that the right-hand sides of equations \eqref{sd} are defined in terms of the known initial data $u_0(x)$ (recall \eqref{Psi1}--\eqref{Psi2}).
Equations \eqref{sd} and items (ii)--(vi) of Proposition \ref{pPsi} yield important analytical and symmetry properties of $a_j(k)$, $j=1,2$, and $b(k)$, summarized in the following proposition.
\begin{proposition}\label{aj-pr}
	The spectral functions $a_j(k)$, $j=1,2$, and $b(k)$, given in \eqref{sd}, satisfy the following properties:
\begin{enumerate}[label=(\roman*)]
	
	\item analyticity:
	$a_{1}(k)$ is analytic in  $k\in\C^{+}$
	and continuous in 
	$\overline{\C^{+}}\setminus\{-B,B\}$;
	$a_{2}(k)$ is analytic in $k\in\C^{-}$
	and continuous in 
	$\overline{\C^{-}}$;
	
	\item behavior as $k\to\infty$:
	$a_1(k)=1+{O}\left(k^{-1}\right)$,
	as $k\rightarrow\infty$, $k\in\overline{\C^+}$;
	$a_2(k)=1+{O}\left(k^{-1}\right)$,
	as $k\rightarrow\infty$, $k\in\overline{\C^-}$;
	and 
	$b(k)={O}\left(k^{-1}\right)$, as $k\rightarrow\infty$, 
	$k\in\R$;
	
	\item symmetries:
	$a_1(k)=\bar{a}_1\left(-\bar{k}\right)$ for  
	$k\in\overline{\C^{+}}\setminus\{-B,B\}$,
	$a_2(k)=\bar{a}_2\left(-\bar{k}\right)$ for
	$k\in\overline{\C^{-}}$,
	and $b(k)=\bar{b}(-k)$ for $k\in\R$;
	
	\item determinant relation:
	$\det S(k)=1$, which is equivalent to
	\begin{equation}\label{detR}
		a_1(k)a_2(k)+b^2(k)=1,\quad
		k\in\R\setminus\{-B,B\};
	\end{equation}
	
	\item behavior as $k\to\pm B$:
	\begin{align}
		\label{a_1+-B}
		&a_1(k)=\frac{A^2a_2(\pm B)}{16(k\mp B)^2}
		+O\left((k\mp B)^{-1}\right),\quad
		k\to\pm B,\quad k\in \overline{\C^{+}},\\
		\label{b+-B}
		&b(k)=-\I\frac{Aa_2(\pm B)}{4(k\mp B)}+O(1),\quad
		k\to\pm B,\quad k\in\R.
	\end{align}
	
\end{enumerate}

\end{proposition}
\begin{proof}
	Taking into account the representations for $a_j(k)$, $j=1,2$, and $b(k)$ given in \eqref{sd}, items (i)--(iv) follow directly from 
	item (ii)--(v) of Proposition \ref{pPsi}.
	To establish item (v), we observe that item (vi) of Proposition \ref{pPsi} yields that
	\begin{equation}\label{a1Bg}
	\begin{split}
		&a_1(k)\sim\frac{1}{(k-B)^2}\det\begin{pmatrix}
			v_1(0,0)& v_2(0,0)\\
			v_2(0,0)& v_1(0,0)
		\end{pmatrix},\quad k\to B,\\
		&a_1(k)\sim\frac{1}{(k+B)^2}\det\begin{pmatrix}
			-\bar{v}_1(0,0)& -\bar{v}_2(0,0)\\
			-\bar{v}_2(0,0)& -\bar{v}_1(0,0)
		\end{pmatrix},\quad k\to-B,
	\end{split}
	\end{equation}
	and
	\begin{equation}\label{bBg}
	\begin{split}
		&b(k)\sim\frac{1}{k-B}\det\begin{pmatrix}
			\frac{4\I}{A}v_2(0,0)& v_1(0,0)\\
			\\[-2ex]
			\frac{4\I}{A}v_1(0,0)& v_2(0,0)
		\end{pmatrix},\quad k\to B,\\
		&b(k)\sim\frac{1}{k+B}\det\begin{pmatrix}
			-\frac{4\I}{A}\bar{v}_2(0,0)& -\bar{v}_1(0,0)\\
			\\[-2ex]
			-\frac{4\I}{A}\bar{v}_1(0,0)& -\bar{v}_2(0,0)
		\end{pmatrix},\quad k\to-B,
	\end{split}
	\end{equation}
	as well as
	\begin{equation}\label{a2Bg}
	\begin{split}
		&a_2(k)\sim\det\begin{pmatrix}
			\frac{4\I}{A}v_2(0,0)& \frac{4\I}{A}v_1(0,0)\\
			\\[-2ex]
			\frac{4\I}{A}v_1(0,0)& \frac{4\I}{A}v_2(0,0)
		\end{pmatrix},\quad k\to B,\\
		&a_1(k)\sim\det\begin{pmatrix}
			-\frac{4\I}{A}\bar{v}_2(0,0)& -\frac{4\I}{A}\bar{v}_1(0,0)\\
			\\[-2ex]
			-\frac{4\I}{A}\bar{v}_1(0,0)& -\frac{4\I}{A}\bar{v}_2(0,0)
		\end{pmatrix},\quad k\to-B.
	\end{split}
	\end{equation}
	Then, \eqref{a1Bg}, \eqref{bBg}, and \eqref{a2Bg} imply \eqref{a_1+-B} and \eqref{b+-B}.
\end{proof}
Invoking the singular behavior of $\Psi_j$, $j=1,2$, near $k=\pm B$ given in item (vi) of Proposition \ref{pPsi}, we can obtain a conservation law for the nmKdV equation having boundary conditions \eqref{bcs} (cf.~\cite[Remark 1]{RS21-DE} and \cite[Remark 3]{XF23}).
\begin{remark}[Conservation law]
	Notice that \eqref{Phi}, \eqref{S}, and item (vi) of Proposition \ref{pPsi} imply the following equation for $a_2(k)$ (cf.~\eqref{sda2}):
	\begin{equation}\label{a2B}
	\begin{split}
		a_2(k)&=\det\left(
		\Psi_2^{(1)}(x,t,k),\Psi_1^{(2)}(x,t,k)
		\right)\\
		&=\frac{16}{A^2}(v_1(x,t)v_1(-x,-t)-v_2(x,t)v_2(-x,-t))
		+O(k-B),\quad k\to B,
	\end{split}
	\end{equation}
	for all $x,t\in\R$, where $v_j(x,t)$, $j=1,2$, are given in \eqref{v1v2-s}.
	Equation \eqref{a2B} yields the following conservation law for \eqref{nmkdvs-a} with boundary conditions \eqref{bcs}:
	\begin{equation}\label{a2cl}
		a_2(B)=v_1(x,t)v_1(-x,-t)-v_2(x,t)v_2(-x,-t),\quad x,t\in\R,
	\end{equation}
	where $v_j(x,t)$, $j=1,2$, can be defined in terms of $u(x,t)$ from \eqref{v1v2-s}.
\end{remark}
\subsection{Pure step initial data}\label{ps-id}
	In this section, we study the spectral functions associated to the
	pure step initial data $u_0(x)$ given in \eqref{sidR}.
	Combining \eqref{S} and \eqref{Phi} with $x=t=0$, we conclude that the scattering matrix $S(k)$ can be calculated as follows
	(recall \eqref{N+-} and \eqref{Psi1}--\eqref{Psi2}):
	\begin{equation*}
	\begin{split}
		S(k)&=\Psi_2^{-1}(0,0,k)\Psi_1(0,0,k)
		=N_+^{-1}(0,0,k)N_-(0,0,k)\\
		&=\begin{pmatrix}
			1+\frac{A^2k^2}{4(k^2-B^2)^2}& \I\frac{Ak}{2(k^2-B^2)}\\
			-\I\frac{Ak}{2(k^2-B^2)}& 1
		\end{pmatrix}.
	\end{split}
	\end{equation*}
	Thus, the spectral functions $a_j(k)$, $j=1,2$, and $b(k)$ corresponding
	to the pure step initial data \eqref{sidR} have the following form (see \eqref{S-m})
	\begin{equation}\label{spf-s}
		a_1(k)=1+\frac{A^2k^2}{4(k^2-B^2)^2},\quad
		a_2(k)=1,\quad b(k)=-\I\frac{Ak}{2(k^2-B^2)}.
	\end{equation}
	In the proposition below we study zeros of $a_1(k)$ defined in \eqref{spf-s} in the closed upper complex half-plane.
	\begin{proposition}\label{a_1s}
		Consider $a_1(k)$ as in \eqref{spf-s}
		Then, depending on the value of the frequency $B>0$, the function
		$a_1(k)$ has the following zeros in $\overline{\C^+}\setminus\{-B,B\}$:
		\begin{enumerate}[label=(\roman*)]
			\item for $0<B<\frac{A}{4}$,
			$a_1(k)$ has two purely imaginary simple zeros at 
			$k=\I k_j$, $j=1,2$, $0<k_1<k_2$, where
			\begin{equation}\label{kj}
				k_j=\frac{1}{4}\left(A+(-1)^j\sqrt{A^2-16B^2}\right),
				\quad j=1,2;
			\end{equation}
			
			\item for $B>\frac{A}{4}$, $a_1(k)$ has two simple zeros at $k=p_1$ and $k=-\bar{p}_1$, where
			\begin{equation}\label{p1}
				p_1=\frac{1}{4}\left(-\sqrt{16B^2-A^2}+\I A\right).
			\end{equation}
			
			\item for $B=\frac{A}{4}$,
			$a_1(k)$ has one purely imaginary zero of multiplicity two
			at $k=\I\frac{A}{4}$;
			
		\end{enumerate}
	\end{proposition}
	\begin{proof}
		Observe that the equation $a_1(k)=0$, see \eqref{spf-s},
		is equivalent to the following 
		biquadratic equation:
		\begin{equation*}
			4k^4+\left(A^2-8B^2\right)k^2+4B^4=0,\quad k\neq\pm B.
		\end{equation*}
		Thus, zeros of $a_1(k)$ can be found as follows:
		\begin{equation}\label{roots}
			k=\frac{\sqrt{2}}{4}\left(
			8B^2-A^2\pm A\left(
			A^2-16B^2
			\right)^{1/2}
			\right)^{1/2},\quad k\neq\pm B.
		\end{equation}
		
		Notice that in the case (i) we have that
		\begin{equation}\label{in-r}
			8B^2-A^2\pm A\sqrt{A^2-16B^2}<0,
		\end{equation}
		which follows from the inequalities
		$8B^2-A^2<0$ and $B\neq0$.
		Combining \eqref{roots} and \eqref{in-r}, we obtain 
		that
		\begin{equation*}
		\begin{split}
			k_j&=\frac{\sqrt{2}}{4}\sqrt{
				A^2-8B^2+(-1)^j A\sqrt{A^2-16B^2}},
				\quad j=1,2,
		\end{split}
		\end{equation*}
		which is equivalent to item (i).

		To prove item (ii), we observe that \eqref{roots} has the following form for $B>\frac{A}{4}$:
		\begin{equation}\label{r3}
			k=\frac{\sqrt{2}}{4}\left(
			8B^2-A^2\pm\I A\sqrt{16B^2-A^2}
			\right)^{1/2}=Be^{\pm\I\frac{\phi_1}{2}+\pi\I n},\quad n\in\Z,
		\end{equation}
		where $\phi_1$ is given by
		\begin{equation}\label{phi1}
			\cos\phi_1=\frac{8B^2-A^2}{8 B^2},\quad
			\sin\phi_1=A\frac{\sqrt{16B^2-A^2}}{8 B^2}.
		\end{equation}
		Using \eqref{phi1}, we conclude that 
		$$
		\cos^2\frac{\phi_1}{2}=\frac{1+\cos\phi_1}{2}
		=\frac{16B^2-A^2}{16B^2},\quad
		\sin^2\frac{\phi_1}{2}=\frac{1-\cos\phi_1}{2}=\frac{A^2}{16B^2},
		$$
		which, together with \eqref{r3}, yields item (ii).
		
		Finally, item (iii) follows directly from \eqref{roots}.
	\end{proof}
	

\subsection{Assumptions on the spectral data}\label{Asssd}
Section \ref{ps-id} studies the spectral data associated to 
to the pure oscillating step initial value \eqref{sidR}.
From \eqref{spf-s} we have that $a_2(k)$ does not have zeros in $\overline{\C^-}$, while Proposition \ref{a_1s} yields that
$a_1(k)$ admits three scenarios of the behavior of zeros in $\overline{\C^+}\setminus\{-B,B\}$ depending on the values of $A$ and $B$:
two simple purely imaginary zeros at $k=\I k_j$, $j=1,2$, with $0<k_1<k_2$,
two simple zeros at $k=p_1$ and $k=-\bar{p}_1$, where 
$\mathrm{Re}\,p_1<0$ and $\mathrm{Im}\,p_1>0$, or
one purely imaginary zero of multiplicity two at $k=\I\ell_1$, $\ell_1>0$.
Motivated by these findings, we will consider
three different classes of initial data,
which, in spectral terms, fall into one of the following three cases:

\begin{description}[itemindent=-\parindent]
	\item[Case I] $a_1(k)$ has two simple purely imaginary zeros in $\overline{\C^+}\setminus\{-B,B\}$ at $k=\I k_1$ and $k=\I k_2$, where $0<k_1<k_2$,
	while $a_2(k)$ has no zeros in $\overline{\C^-}$;
	
	\item[Case II] $a_1(k)$ has two simple zeros in $\overline{\C^+}\setminus\{-B,B\}$ at $k=p_1$ and $k=-\bar{p}_1$
	with $\mathrm{Re}\,p_1<0$ and $\mathrm{Im}\,p_1>0$,
	while $a_2(k)$ has no zeros in $\overline{\C^-}$;
	
	\item[Case III] $a_1(k)$ has one double zero in $\overline{\C^+}\setminus\{-B,B\}$ at $k=\I\ell_1$, $\ell_1>0$,
	and
	$a_2(k)$ has no zeros in $\overline{\C^-}$.
\end{description}

Recall that the singularity rate of $a_1(k)$
near the points $k=\pm B$ is qualitatively different for $a_2(\pm B)\neq0$ and
$a_2(\pm B)=0$,
see \eqref{a_1+-B}.
This motivates us 
to specify three more cases, where
the spectral function $a_1(k)$ satisfies the same assumptions on zeros as in
Cases I--III, while $a_2(k)$ has simple zeros at $k=\pm B$.
Also, observing that \eqref{detR} yields
\begin{equation}\label{b(B)}
	\lim\limits_{k\to B}(k\mp B)a_1(k)=\frac{1-b^2(\pm B)}{a_2^\prime(\pm B)},
	\quad \mbox{if }\,a_2(\pm B)=0,\,\,
	a_2^\prime(\pm B)\neq0,
\end{equation}
we make an assumption $b(B)\neq \pm 1$ (recall that $b(B)=\bar{b}(-B)$).
Thus, in addition to Cases I--III described above, we consider the following associated scattering data:
\begin{description}[itemindent=-\parindent]
	\item[Case $\widetilde{\mathrm{\mathbf{I}}}$] $a_1(k)$ has two simple purely imaginary zeros in $\overline{\C^+}\setminus\{-B,B\}$ at $k=\I k_1$ and $k=\I k_2$, where $0<k_1<k_2$;
	$a_2(k)$ has two simple zeros in $\overline{\C^-}$ at $k=\pm B$,
	and $b(B)\neq\pm 1$;
	
	\item[Case $\widetilde{\mathrm{\mathbf{II}}}$] $a_1(k)$ has two simple zeros in $\overline{\C^+}\setminus\{-B,B\}$ at $k=p_1$ and $k=-\bar{p}_1$
	with $\mathrm{Re}\,p_1<0$ and $\mathrm{Im}\,p_1>0$; $a_2(k)$ has two simple zeros in $\overline{\C^-}$ at $k=\pm B$,
	and $b(B)\neq\pm 1$;
	
	\item[Case $\widetilde{\mathrm{\mathbf{III}}}$] $a_1(k)$ has one double zero in $\overline{\C^+}\setminus\{-B,B\}$ at $k=\I\ell_1$, $\ell_1>0$;
	$a_2(k)$ has two simple zeros in $\overline{\C^-}$ at $k=\pm B$,
	and $b(B)\neq\pm 1$.
\end{description}

Exploiting the determinant relation \eqref{detR} and the singular behavior at $k=\pm B$ of $a_1(k)$, see \eqref{a_1+-B},
we show that the functions $a_j(k)$, $j=1,2$, are fully determined by the spectral function $b(k)$ in Cases I--III
(cf.~\cite[Proposition 3]{RS21-DE} and \cite[Proposition 3]{XF23}).
Moreover, the function $b(k)$ itself cannot be taken arbitrary and it must satisfy additional constrains, see \eqref{arg-B} and \eqref{sdj} below.
\begin{proposition}[$a_j(k)$ in terms of $b(k)$, Cases I--III]
	\label{za1a2}
	Suppose that the spectral functions $a_j(k)$, $j=1,2$, satisfy assumptions described in either 
	Case $\mathrm{I}$, $\mathrm{II}$, or $\mathrm{III}$.
	Introduce, in terms of $b(k)$, the constants 
	$\varphi_j$, $j=1,2$, as follows:
	\begin{equation}\label{vphi+}
		\varphi_1=\frac{1}{\pi\I}
		\mathrm{v.p.}\int_{-\infty}^\infty
		\frac{\log\left[\left(\frac{\zeta^2-B^2}{\zeta^2+1}\right)^2
			\left(1-b^2(\zeta)\right)\right]}
		{\zeta-B}\,d\zeta,
	\end{equation}
	and
	\begin{equation}\label{vphi2}
		e^{\I\varphi_2}=\frac{(B+\I)^4}{\left(B^2+1\right)^2}
		e^{\I\mathrm{Im}\,\varphi_1},\quad
		\varphi_2\in[0,2\pi).
	\end{equation}
	Also, we define the following constants in terms of $\varphi_1$ and $\varphi_2$:
	\begin{equation}\label{d1d2}
		\begin{split}
			&d_1=\frac{A}{8}\sqrt{2(1-\cos\varphi_2)}\,
			e^{-\mathrm{Re}\,\frac{\varphi_1}{2}},\\
			&d_2=d_1^2-B^2+\frac{\sqrt{2}AB
				e^{-\mathrm{Re}\,\frac{\varphi_1}{2}}}
			{4\sqrt{1-\cos\varphi_2}}\sin\varphi_2,\quad
			\mbox{for }\,\varphi_2\neq0.
		\end{split}
	\end{equation}
	Then we have that
	\begin{enumerate}[label=(\roman*)]
		\item $\varphi_2\neq0$ in Cases $\mathrm{I}$--$\mathrm{III}$; this condition reads	in terms of $\varphi_1$ as follows:
		\begin{equation}\label{arg-B}
			4\arg(B+\I)+\mathrm{Im}\,\varphi_1\neq 2\pi n,\quad n\in\Z;
		\end{equation}
		
		\item $d_2$ satisfies the following conditions:
		\begin{equation}\label{sdj}
			\mbox{Case $\mathrm{I}$:}\quad 0<d_2<d_1^2,\quad
			\mbox{Case $\mathrm{II}$:}\quad d_2<0,\quad
			\mbox{Case $\mathrm{III}$:}\quad d_2=0;
		\end{equation}
		
		\item zeros of $a_1(k)$ are determined as follows
		(recall \eqref{d1d2} and \eqref{sdj}):
		\begin{equation*}
		\begin{split}
			&\mbox{Case $\mathrm{I}$:}\quad 
			k_j=d_1+(-1)^j\sqrt{d_2},\quad j=1,2,\\
			&\mbox{Case $\mathrm{II}$:}\quad
			p_1=-\sqrt{-d_2}+\I d_1,\\
			&\mbox{Case $\mathrm{III}$:}\quad
			\ell_1=d_1;
		\end{split}
		\end{equation*}
		
		\item $a_j(k)$, $j=1,2$, are determined in terms of $b(k)$ as follows
		(the trace formulae):
		\begin{subequations}
		\begin{align*}
			&\mbox{Case $\mathrm{I}$:}\quad a_1(k)=
			\frac{(k-\I k_1)(k-\I k_2)(k+\I)^2}{\left(k^2-B^2\right)^2}
			e^{\varphi(k)},\quad
			a_2(k)=\frac{(k-\I)^2}{(k-\I k_1)(k-\I k_2)}e^{-\varphi(k)},\\
			&\mbox{Case $\mathrm{II}$:}\quad
			a_1(k)=
			\frac{(k-p_1)(k+ \bar{p}_1)(k+\I)^2}{\left(k^2-B^2\right)^2}
			e^{\varphi(k)},\quad
			a_2(k)=\frac{(k-\I)^2}{(k-p_1)(k+\bar{p}_1)}e^{-\varphi(k)},\\
			&\mbox{Case $\mathrm{III}$:}\quad
			a_1(k)=
			\frac{(k-\I\ell_1)^2(k+\I)^2}{\left(k^2-B^2\right)^2}
			e^{\varphi(k)},\quad
			a_2(k)=\frac{(k-\I)^2}{(k-\I\ell_1)^2}e^{-\varphi(k)},
		\end{align*}
	\end{subequations}
		where 
		\begin{equation}\label{vphi}
			\varphi(k)=\frac{1}{2\pi\I}
			\int_{-\infty}^\infty
			\frac{\log\left[\left(\frac{\zeta^2-B^2}{\zeta^2+1}\right)^2
				\left(1-b^2(\zeta)\right)\right]}
			{\zeta-k}\,d\zeta,\quad k\in\C\setminus\R.
		\end{equation}
	\end{enumerate}
	
\end{proposition}
\begin{proof}
	\textbf{Step 1.}
	Let us denote zeros of $a_1(k)$ by $z_1$ and $z_2$,
	that is 
	\begin{equation}\label{z1z2}
	\begin{split}
		&(z_1,z_2)=(\I k_1, \I k_2),\,\,
		\mbox{in Case I,}\quad
		(z_1,z_2)=(p_1,-\bar{p}_1),\,\,
		\mbox{in Case II,}\\
		&z_1=z_2=\I\ell_1,\,\,\mbox{in Case III}.
	\end{split}
	\end{equation}
	Notice that in Cases I--III we have:
	\begin{equation}\label{Iz1z2}
		(z_1+z_2)\in\I\R,\quad\mathrm{Im}\,(z_1+z_2)>0.
	\end{equation}
	Introduce the following functions (recall \eqref{a_1+-B}):
	\begin{equation}\label{daj-til}
		\tilde{a}_1(k)=
		\frac{\left(k^2-B^2\right)^2}{(k-z_1)(k-z_2)(k+\I)^2}a_1(k),\quad
		\tilde{a}_2(k)=
		\frac{(k-z_1)(k-z_2)}{(k-\I)^2}a_2(k).
	\end{equation}
	Then the determinant relation, see Proposition \ref{aj-pr}, item (iv),
	can be viewed as a scalar Riemann-Hilbert problem for $\tilde{a}_1(k)$
	and $\tilde{a}_2^{-1}(k)$ with the jump across $\R$ oriented from left to right:
	\begin{equation*}
		\begin{split}
			&\tilde{a}_1(k)=\tilde{a}_2^{-1}(k)
			\frac{\left(k^2-B^2\right)^2}{\left(k^2+1\right)^2}
			\left(1-b^2(k)\right),\quad k\in\R,\\
			&\tilde{a}_j(k)\to 1,\quad j=1,2,\quad k\to\infty.
		\end{split}
	\end{equation*}
	Applying the Plemelj-Sokhotski formula, we obtain the following representations for $\tilde{a}_j(k)$, $j=1,2$:
	\begin{equation}\label{aj-til}
		\tilde{a}_1(k)=e^{\varphi(k)},\quad k\in\C^+,
		\quad\text{and}\quad
		\tilde{a}_2(k)=e^{-\varphi(k)},\quad k\in\C^-,
	\end{equation}
	where $\varphi(k)$ is given in \eqref{vphi}.

	\medskip
	
	\textbf{Step 2.}
	Equations \eqref{daj-til} and \eqref{aj-til}
	imply the following behavior 
	of $a_1(k)$ as $k\to\pm B$:
	\begin{equation}\label{pa_1+-B}
		a_1(k)=\frac{(B\mp z_1)(B\mp z_2)(B\pm \I)^2}
		{4B^2(k\mp B)^2}e^{\varphi(\pm B+\I0)}
		+O\left((k\mp B)^{-1}\right),\quad
		k\to\pm B,\quad k\in \overline{\C^{+}},
	\end{equation}
	as well as that
	\begin{equation}\label{pa_2+-B}
		a_2(\pm B)=\frac{(B\mp\I)^2}{(B\mp z_1)(B\mp z_2)}
		e^{-\varphi(\pm B-\I0)}.
	\end{equation}
	Combining \eqref{a_1+-B}, \eqref{pa_1+-B}, and \eqref{pa_2+-B},
	we obtain the following system of equations for $z_j$, $j=1,2$:
	\begin{equation}\label{syseq}
		\begin{split}
			&(B-z_1)^2(B-z_2)^2=
			\frac{A^2B^2(B-\I)^4}{4\left(B^2+1\right)^2}e^{-\varphi_1},\\
			&(B+z_1)^2(B+z_2)^2=
			\frac{A^2B^2(B+\I)^4}{4\left(B^2+1\right)^2}e^{-\tilde{\varphi}_1},
		\end{split}
	\end{equation}
	where, as per \eqref{vphi+} (see \eqref{vphi})
	\begin{equation}\label{tvp}
		\begin{split}
			&\varphi_1=\varphi(B+\I0)+\varphi(B-\I0)=\frac{1}{\pi\I}
			\mathrm{v.p.}\int_{-\infty}^\infty
			\frac{\log\left[\left(\frac{\zeta^2-B^2}{\zeta^2+1}\right)^2
				\left(1-b^2(\zeta)\right)\right]}
			{\zeta-B}\,d\zeta,\\
			&\tilde{\varphi}_1=\varphi(-B+\I0)+\varphi(-B-\I0)=\frac{1}{\pi\I}
			\mathrm{v.p.}\int_{-\infty}^\infty
			\frac{\log\left[\left(\frac{\zeta^2-B^2}{\zeta^2+1}\right)^2
				\left(1-b^2(\zeta)\right)\right]}
			{\zeta+B}\,d\zeta.
		\end{split}
	\end{equation}
	The symmetry condition $\bar{b}(k)=b(-k)$
	(see Proposition \ref{aj-pr}, item (iii)) yields that
	\begin{equation}\label{svphi}
		\tilde{\varphi}_1=\bar\varphi_1.
	\end{equation}
	
	Adding and subtracting equations in \eqref{syseq}, we arrive at the following system (recall \eqref{svphi}):
	\begin{equation}\label{sz1z2}
		\begin{split}
			&\left(B^2+z_1z_2\right)^2+B^2(z_1+z_2)^2
			=\frac{A^2B^2C_+}{8(B^2+1)^2},\\
			&(z_1+z_2)(B^2+z_1z_2)
			=\frac{A^2BC_-}{16\left(B^2+1\right)^2},
		\end{split}
	\end{equation}
	with
	\begin{equation}\label{C+-}
		C_\pm=(B+\I)^4e^{-\bar\varphi_1}\pm(B-\I)^4e^{-\varphi_1}.
	\end{equation}
	Since $(z_1+z_2)\neq 0$, see \eqref{z1z2}, we can express $(B^2+z_1z_2)$ in terms of $(z_1+z_2)^{-1}$ from the second equation in \eqref{sz1z2}.
	Thus, the first equation in \eqref{sz1z2} yields the following biquadratic equation for $(z_1+z_2)$:
	\begin{equation}\label{biz1z2}
		(z_1+z_2)^4-
		\frac{A^2C_+}{8(B^2+1)^2}(z_1+z_2)^2
		+\frac{A^4C_-^2}{256\left(B^2+1\right)^4}=0.
	\end{equation}
	Using that (recall \eqref{C+-})
	$$
	C_+^2-C_-^2=4\left(B^2+1\right)^4e^{-2\mathrm{Re}\,\varphi_1},
	$$
	we conclude that the discriminant of \eqref{biz1z2} is equal to
	$\frac{A^4}{16}e^{-2\mathrm{Re}\,\varphi_1}$.
	Taking into account \eqref{Iz1z2} and that
	$$
	C_+=2\mathrm{Re}\left((B+\I)^4e^{-\bar{\varphi}_1}\right)
	=2e^{-\mathrm{Re}\,\varphi_1}
	\mathrm{Re}\left((B+\I)^4e^{\I\mathrm{Im}\,\varphi_1}\right),
	$$
	we obtain the following expression for $(z_1+z_2)^2$:
	\begin{equation}\label{z1+z2}
		(z_1+z_2)^2=\frac{A^2}{8}e^{-\mathrm{Re}\,\varphi_1}
		\left(
		\frac{\mathrm{Re}\left((B+\I)^4e^{\I\mathrm{Im}\,\varphi_1}\right)}
		{(B^2+1)^2}-1\right)
		=\frac{A^2}{8}
		(\cos\varphi_2-1)e^{-\mathrm{Re}\,\varphi_1},
	\end{equation}
	where we have used \eqref{vphi2}. 
	\medskip
	
	\textbf{Proof of item (i).}
	Applying \eqref{Iz1z2}, equation \eqref{z1+z2} implies that
	$\cos\varphi_2\neq1$ (see \eqref{vphi2}):
	\begin{equation}
		\label{b-r}
		\frac{\mathrm{Re}\,\left(
			(B+\I)^4e^{\I\mathrm{Im}\,\varphi_1}\right)}
		{\left(B^2+1\right)^2}\neq 1.
	\end{equation}
	Since $\left|\frac{\mathrm{Re}\,\left(
		(B+\I)^4e^{\I\mathrm{Im}\,\varphi_1}\right)}
	{\left(B^2+1\right)^2}\right|=1$, inequality \eqref{b-r} yields
	\begin{equation*}
		\frac{(B+\I)^2}
		{B^2+1}e^{\I\frac{\mathrm{Im}\,\varphi_1}{2}}\neq\pm1,
	\end{equation*}
	which is equivalent to \eqref{arg-B}.
	\medskip
	
	\textbf{Proof of items (ii) and (iii) in Case I.}
	Since $(z_1+z_2)=\I(k_1+k_2)$, we obtain from \eqref{z1+z2} that
	\begin{equation}\label{k1+k2}
		k_1+k_2=2d_1,
	\end{equation}
	where $d_1$ is given in \eqref{d1d2}.
	Using that $z_1z_2=-k_1k_2$ and (see \eqref{C+-})
	$$
	C_-=2\I\mathrm{Im}\left((B+\I)^4e^{-\bar{\varphi}_1}\right)
	=2\I e^{-\mathrm{Re}\,\varphi_1}
	\mathrm{Im}\left((B+\I)^4e^{\I\mathrm{Im}\,\varphi_1}\right),
	$$
	we obtain from the second equation in \eqref{sz1z2} the following expression for the product $k_1k_2$ (see \eqref{vphi2}):
	\begin{equation}\label{k1k2}
		\begin{split}
			k_1k_2&=B^2-\frac{A^2Be^{-\mathrm{Re}\,\varphi_1}}
			{16d_1(B^2+1)^2}
			\mathrm{Im}\left((B+\I)^4e^{\I\mathrm{Im}\,\varphi_1}\right)\\
			&=B^2-\frac{ABe^{-\mathrm{Re}\,\frac{\varphi_1}{2}}}
			{2\sqrt{2(1-\cos\varphi_2)}}\sin\varphi_2
			=d_1^2-d_2,
		\end{split}
	\end{equation}
	where $d_j$, $j=1,2$, are given in \eqref{d1d2}.
	Combining \eqref{k1+k2} and \eqref{k1k2}, we arrive at the following quadratic equation for $k_1$:
	\begin{equation}\label{qk1}
		k_1^2-2d_1k_1+d_1^2-d_2=0.
	\end{equation}
	Taking into account that equation \eqref{qk1} must have two distinct strictly positive zeros, we have items (ii,iii) in Case I.
	\medskip
	
	\textbf{Proof of items (ii) and (iii) in Case II.}
	Using that $z_1+z_2=2\I\mathrm{Im}\,p_1$, equation \eqref{z1+z2} immediately implies that 
	\begin{equation}\label{Ip1}
		\mathrm{Im}\,p_1=d_1.
	\end{equation}
	Arguing similarly as in \eqref{k1k2}, we obtain from the second equation in \eqref{sz1z2} that 
	$|p_1|^2=d_1^2-d_2$ which, together with \eqref{Ip1},
	yields items (ii,iii) in Case II.
	\medskip
	
	\textbf{Proof of items (ii) and (iii) in Case III.}
	Observing that $z_1+z_2=2\I\ell_1$, equation \eqref{z1+z2} yields 
	item (iii) in Case III.
	Substituting $z_1=z_2=\I d_1$ into the second equation in \eqref{sz1z2} and
	arguing as in \eqref{k1k2}, we obtain item (ii) in Case III.
	\medskip 
	
	\textbf{Proof of item (iv).}
	Combining \eqref{z1z2}, \eqref{daj-til}, and \eqref{aj-til}, we arrive at item (iv) of the proposition.
	\end{proof}
	
	\begin{remark}\label{Rps}
	Let us verify Proposition \ref{za1a2} for the spectral functions \eqref{spf-s} associated to the pure oscillating step initial \eqref{sidR}.
	Observe that (see \eqref{spf-s} and \ref{vphi})
	\begin{equation}\label{b1e}
		\left(\frac{\zeta^2-B^2}{\zeta^2+1}\right)^2
		\left(1-b^2(\zeta)\right)
		=\frac{4k^4+(A^2-8B^2)k^2+4B^4}{4(k^2+1)^2}.
	\end{equation}
	The
	polynomial $4k^4+(A^2-8B^2)k^2+4B^4$ can be written as follows
	(see the proof of Proposition \ref{a_1s}):
	\begin{equation}\label{biq}
		4k^4+(A^2-8B^2)k^2+4B^4=4(k-z_1)(k-z_2)(k-\bar{z}_1)(k-\bar{z}_2),
	\end{equation}
	where $(z_1,z_2)$ are defined by \eqref{z1z2} 
	with $k_j$, $j=1,2$, and $p_1$ given as 
	in \eqref{kj} and \eqref{p1}, respectively, and 
	with $\ell_1=\frac{A}{4}$.
	
	Combining \eqref{vphi}, \eqref{b1e}, \eqref{biq},
	and using the Plemelj-Sokhotski formula,
	we conclude that
	$e^{\varphi(k)}$ satisfies the following scalar Riemann-Hilbert problem:
	\begin{equation}\label{RHsiv}
		\begin{split}
			&\left(e^{\varphi(k)}\right)_+=
			\left(e^{\varphi(k)}\right)_-
			\frac{(k-z_1)(k-z_2)(k-\bar{z}_1)(k-\bar{z}_2)}
			{(k^2+1)^2},\quad k\in\R,\\
			&e^{\varphi(k)}\to1,\quad k\to\infty.
		\end{split}
	\end{equation}	
	
	Direct computations show that a unique solution of the Riemann-Hilbert problem \eqref{RHsiv} reads as follows:
	\begin{equation}\label{vphs}
		e^\varphi(k)=
		\begin{cases}
			\frac{(k-\bar{z}_1)(k-\bar{z}_2)}{(k+\I)^2},&k\in\C^+,\\
			\frac{(k-\I)^2}{(k-z_1)(k-z_2)},&k\in\C^-.
		\end{cases}
	\end{equation}
	Taking into account that (see \eqref{kj}, \eqref{p1}, and \eqref{z1z2})
	\begin{equation}\label{Bz1}
		(B-z_1)(B-z_2)=-\I\frac{AB}{2},
		\quad\mbox{in Cases $\mathrm{I}$--$\mathrm{III}$},
	\end{equation}
	and using the first equation in \eqref{tvp}, \eqref{vphs}, and \eqref{Bz1}, we arrive at the following representation for $e^{\varphi_1}$:
	\begin{equation}\label{vph1s}
		e^{\varphi_1}=
		e^{\varphi(B+\I0)}e^{\varphi(B-\I0)}=
		-\frac{(B-\I)^4}{(B^2+1)^2},
		\quad\mbox{in Cases $\mathrm{I}$--$\mathrm{III}$}.
	\end{equation}
	
	Equation \eqref{vph1s} yields that $\mathrm{Re}\,\varphi_1=0$.
	Thus, combining \eqref{vph1s} and \eqref{vphi2}, we conclude that $\varphi_2=\pi$.
	Using \eqref{d1d2}, we obtain the following expressions for $d_j$, $j=1,2$:
	$$
	d_1=\frac{A}{4},\quad d_2=\frac{A^2}{16}-B^2.
	$$
	Recalling that $0<B<\frac{A}{4}$ in Case $\mathrm{I}$,
	$B>\frac{A}{4}$ in Case $\mathrm{II}$, 
	and $B=\frac{A}{4}$ in Case $\mathrm{III}$, we establish the validity of 
	Proposition \ref{za1a2} for the pure oscillating step initial data. 
	
	\end{remark}

Applying similar arguments as in Proposition \ref{za1a2},
we derive expressions for the spectral functions $a_j(k)$, $j=1,2$,
particularly for their zeros,
in terms of $b(k)$ in 
Cases $\widetilde{\mathrm{I}}$--$\widetilde{\mathrm{III}}$.


\begin{proposition}[$a_j(k)$ in terms of $b(k)$ in Cases $\widetilde{\mathrm{I}}$--$\widetilde{\mathrm{III}}$]
	\label{tza1}
Suppose that the spectral functions $a_j(k)$, $j=1,2$, and $b(k)$ satisfy 
assumptions described in either Case $\widetilde{\mathrm{I}}$,
$\widetilde{\mathrm{II}}$, or $\widetilde{\mathrm{III}}$.
Introduce, in terms of $b(k)$, the following constants:
\begin{equation}\label{E+-}
	E_{\pm}=\I\frac{AB}{2E_1E_2}\left(
	b(B)\pm\sqrt{E_1^2+b^2(B)}\right),
\end{equation}
with (recall that $b(B)\neq\pm1$ in 
Cases $\widetilde{\mathrm{I}}$--$\widetilde{\mathrm{III}}$ by assumption)
\begin{equation}\label{E12}
	E_1=\exp\left(\frac{1}{2\pi\I}
	\mathrm{v.p.}\int_{-\infty}^\infty
	\frac{\log\left(1-b^2(\zeta)\right)}
	{\zeta-B}\,d\zeta\right),\quad
	E_2=\exp\left(\frac{1}{2}\log\left(1-b^2(B)\right)\right).
\end{equation}
Then we have that
\begin{enumerate}[label=(\roman*)]
	\item either $E_+$, $E_-$, or both $E_+$ and $E_-$ 
	satisfy the following two conditions
	in Cases $\widetilde{\mathrm{I}}$--$\widetilde{\mathrm{III}}$:
	\begin{subequations}
		\label{E+-c}
	\begin{align}
		\label{E+-ca}
		&\mbox{Case $\widetilde{\mathrm{I}}$:}
		\quad 0<B^2-\mathrm{Re}\,E_\pm<
		\frac{\mathrm{Im}^2\,E_\pm}{4B^2}\,
		\mbox{ and }\,\mathrm{Im}\,E_\pm<0,\\
		&\mbox{Case $\widetilde{\mathrm{II}}$:}
		\quad \frac{\mathrm{Im}^2\,E_\pm}{4B^2}<B^2-\mathrm{Re}\,E_\pm\,
		\mbox{ and }\,\mathrm{Im}\,E_\pm<0,\\
		&\mbox{Case $\widetilde{\mathrm{III}}$:}
		\quad \frac{\mathrm{Im}^2\,E_\pm}{4B^2}=B^2-\mathrm{Re}\,E_\pm\,
		\mbox{ and }\,\mathrm{Im}\,E_\pm<0;
	\end{align}
	\end{subequations}
	
	\item zeros of $a_1(k)$ have the following representation
	(equations below are fulfilled with either $E_+$ or $E_-$ in place of $E_\pm$; this constant, used in these expressions, must satisfy \eqref{E+-c} in the corresponding case):
	\begin{equation*}
	\begin{split}
		&\mbox{Case $\widetilde{\mathrm{I}}$:}\quad
		k_j=-\frac{\mathrm{Im}\,E_\pm}{2B}
		+(-1)^j\sqrt{\frac{\mathrm{Im}^2\,E_\pm}{4B^2}
		+\mathrm{Re}\,E_\pm-B^2},\quad j=1,2,\\
		&\mbox{Case $\widetilde{\mathrm{II}}$:}\quad
		p_1=-\sqrt{B^2-\mathrm{Re}\,E_\pm-\frac{\mathrm{Im}^2\,E_\pm}{4B^2}}
			-\I \frac{\mathrm{Im}\,E_\pm}{2B},\\
		&\mbox{Case $\widetilde{\mathrm{III}}$:}\quad
		\ell_1=-\frac{\mathrm{Im}\,E_\pm}{2B};
	\end{split}
	\end{equation*}
	
	\item $a_j(k)$, $j=1,2$, are determined in terms of $b(k)$ as follows
	(the trace formulae):
	\begin{equation*}
		\begin{split}
			&\mbox{Case $\widetilde{\mathrm{I}}$:}\quad a_1(k)=
			\frac{(k-\I k_1)(k-\I k_2)}{k^2-B^2}
			e^{\psi(k)},\quad
			a_2(k)=\frac{k^2-B^2}{(k-\I k_1)(k-\I k_2)}e^{-\psi(k)},\\
			&\mbox{Case $\widetilde{\mathrm{II}}$:}\quad
			a_1(k)=
			\frac{(k-p_1)(k+ \bar{p}_1)}{k^2-B^2}
			e^{\psi(k)},\quad
			a_2(k)=\frac{k^2-B^2}{(k-p_1)(k+\bar{p}_1)}e^{-\psi(k)},\\
			&\mbox{Case $\widetilde{\mathrm{III}}$:}\quad
			a_1(k)=
			\frac{(k-\I\ell_1)^2}{k^2-B^2}
			e^{\psi(k)},\quad
			a_2(k)=\frac{k^2-B^2}{(k-\I\ell_1)^2}e^{-\psi(k)},
		\end{split}
	\end{equation*}
	where 
	\begin{equation}
		\label{vpsi}
		\psi(k)=\frac{1}{2\pi\I}
		\int_{-\infty}^\infty
		\frac{\log\left(1-b^2(\zeta)\right)}
		{\zeta-k}\,d\zeta,\quad k\in\C\setminus\R.
	\end{equation}
\end{enumerate}
\end{proposition}
\begin{proof}
	\textbf{Step 1.}
	Proposition \ref{pPsi}, item (vi) implies that $\Psi_1$ has the following behavior as $k\to B$:
	\begin{equation}\label{Psi1B}
	\begin{split}
		&\Psi_1^{(1)}(x,t,k)=\frac{1}{k-B}
		\begin{pmatrix}v_1(x,t)\\v_2(x,t)\end{pmatrix}+
		\begin{pmatrix}s_1(x,t)\\s_2(x,t)\end{pmatrix}+
		O(k-B),
		\quad k\to B,\\
		&\Psi_1^{(2)}(x,t,k)=\frac{4\I}{A}
		e^{-2\I Bx-8\I B^3t}
		\begin{pmatrix}v_1(x,t)\\v_2(x,t)\end{pmatrix}
		+(k-B)\begin{pmatrix}h_1(x,t)\\h_2(x,t)\end{pmatrix}
		+O\left((k-B)^2\right),\quad
		k\to B,
	\end{split}
	\end{equation}
	with some functions $s_j(x,t)$ and $h_j(x,t)$, $j=1,2$.
	Using the first symmetry relation in \eqref{Psi-sym} as well as recalling item (vi) of Proposition \ref{pPsi}, we conclude that
	$\Psi_2$ admits the following expansion as $k\to B$:
	\begin{equation}\label{Psi2B}
	\begin{split}
		&\Psi_2^{(1)}(x,t,k)=\frac{4\I}{A}
		e^{2\I Bx+8\I B^3t}
		\begin{pmatrix}v_2(-x,-t)\\v_1(-x,-t)\end{pmatrix}
		+(k-B)\begin{pmatrix}h_2(-x,-t)\\h_1(-x,-t)\end{pmatrix}
		+O\left((k-B)^2\right),
		\quad k\to B,\\
		&\Psi_2^{(2)}(x,t,k)=\frac{1}{k-B}
		\begin{pmatrix}v_2(-x,-t)\\v_1(-x,-t)\end{pmatrix}
		+\begin{pmatrix}s_2(-x,-t)\\s_1(-x,-t)\end{pmatrix}
		+O(k-B),
		\quad k\to B.
	\end{split}
	\end{equation}
	Combining \eqref{sd}, \eqref{Psi1B}, and \eqref{Psi2B}, we arrive at the following
	approximations for $a_j(k)$, $j=1,2$, and $b(k)$:
	\begin{subequations}
		\label{a1a2bB}
	\begin{align}
		&a_1(k)=\frac{1}{(k-B)^2}\left(v_1^2-v_2^2\right)(0,0)
		+\frac{2}{k-B}(v_1s_1-v_2s_2)(0,0)+O(1),
		\quad k\to B,\\
		&a_2(k)=\frac{16}{A^2}\left(v_1^2-v_2^2\right)(0,0)
		+\frac{8\I}{A}(k-B)(v_2h_2-v_1h_1)(0,0)
		+O\left((k-B)^2\right),\quad k\to B,\\
		&b(k)=\frac{4\I}{A(k-B)}\left(v_2^2-v_1^2\right)(0,0)
		+\frac{4\I}{A}(v_2s_2-v_1s_1)(0,0)
		+(v_2h_2-v_1h_1)(0,0)+O(k-B),\quad k\to B.
	\end{align}
\end{subequations}
	Since $a_2(B)=0$, we have that $v_1^2(0,0)=v_2^2(0,0)$.
	Therefore, \eqref{a1a2bB} implies the following expansion for $a_1(k)$
	as $k\to B$ (cf.\,\,\cite[Equation (2.45)]{RS21-DE} and 
	\cite[Equation (A.14)]{XF23}):
	\begin{equation}\label{a1B1}
		a_1(k)=\I\frac{A}{2(k-B)}\left(
		b(B)+\I\frac{A}{8}a_2^\prime(B)\right)+O(1),\quad k\to B.
	\end{equation}
	\medskip
	
	\textbf{Step 2.}
	Introduce the following functions (cf.\,\,\eqref{daj-til})
	\begin{equation}\label{ch-a1}
		\check{a}_1(k)=\frac{k^2-B^2}{(k-z_1)(k-z_2)}a_1(k),\quad
		\check{a}_2(k)=\frac{(k-z_1)(k-z_2)}{k^2-B^2}a_2(k),
	\end{equation}
	with, as per \eqref{z1z2},
	\begin{equation}\label{z1z2-t}
		\begin{split}
			&(z_1,z_2)=(\I k_1, \I k_2),\,\,
			\mbox{in Case $\widetilde{\mathrm{I}}$,}\quad
			(z_1,z_2)=(p_1,-\bar{p}_1),\,\,
			\mbox{in Case $\widetilde{\mathrm{II}}$,}\\
			&z_1=z_2=\I\ell_1,\,\,\mbox{in Case $\widetilde{\mathrm{III}}$}.
		\end{split}
	\end{equation}
	Using determinant relation \eqref{detR}, we conclude
	that $\check{a}_1$ and $\check{a}_2$ have the following representations
	(cf.\,\,\eqref{aj-til}):
	\begin{equation}\label{ch-p-a1}
		\check{a}_1(k)=e^{\psi(k)},\quad k\in\C^+,
		\quad\text{and}\quad
		\check{a}_2(k)=e^{-\psi(k)},\quad k\in\C^-,
	\end{equation}
	where $\psi(k)$ is given in \eqref{vpsi}.
	Combining \eqref{a1B1}, \eqref{ch-a1}, and \eqref{ch-p-a1},
	we obtain the following biquadratic equation for the product $(B-z_1)(B-z_2)$ (recall \eqref{E12}):
	\begin{equation}\label{bi-B}
		E_1E_2(B-z_1)^2(B-z_2)^2-\I ABb(B)(B-z_1)(B-z_2)
		+\frac{A^2B^2}{4}E_1E_2^{-1}=0.
	\end{equation}
	Equation \eqref{bi-B} implies that
	\begin{equation}\label{Bz1z2}
		(B-z_1)(B-z_2)=E_+,\quad\mbox{or}\quad
		(B-z_1)(B-z_2)=E_-,
	\end{equation}
	where $E_\pm$ are given in \eqref{E+-}.
	\medskip
	
	\textbf{Proof of items (i) and (ii) in Case $\widetilde{\mathrm{\mathbf{I}}}$.}
	Recalling that $z_j=\I k_j$, $j=1,2$, $0<k_1<k_2$, we obtain from
	\eqref{Bz1z2} the following equations:
	\begin{equation}\label{Bk1k2}
		B^2-k_1k_2=\mathrm{Re}\, E_\pm,\quad
		B(k_1+k_2)=-\mathrm{Im}\,E_\pm.
	\end{equation}
	The second equation in \eqref{Bk1k2} implies that
	$\mathrm{Im}\,E_\pm<0$, see the second condition in \eqref{E+-ca}.
	Moreover, \eqref{Bk1k2} yields the following quadratic equation for $k_1$:
	\begin{equation}\label{quk1}
		k_1^2+\frac{\mathrm{Im}\,E_\pm}{B}k_1+B^2-\mathrm{Re}\,E_\pm=0.
	\end{equation}
	Solving equation \eqref{quk1}, we obtain the first condition in 
	\eqref{E+-ca} and item (ii) in Case $\widetilde{\mathrm{I}}$.
	\medskip
	
	\textbf{Proof of items (i) and (ii) in Case $\widetilde{\mathrm{\mathbf{II}}}$.}
	Using that $z_1=p_1$ and $z_2=-\bar{p}_1$ with $\mathrm{Im}\,p_1>0$ and
	$\mathrm{Re}\,p_1<0$, we conclude from \eqref{Bz1z2} that
	\begin{equation*}
		B^2-|p_1|^2=\mathrm{Re}\,E_\pm,\quad
		\mathrm{Im}\,p_1=-\frac{\mathrm{Im}\,E_\pm}{2B},
	\end{equation*}
	which yield items (i) and (ii) in Case $\widetilde{\mathrm{II}}$.
	\medskip
	
	\textbf{Proof of items (i) and (ii) in Case $\widetilde{\mathrm{\mathbf{III}}}$.}
	Since $z_1=z_2=\I\ell_1$, $\ell_1>0$, we have from \eqref{Bz1z2} that
	$(B-\I\ell_1)^2=E_\pm$.
	This implies the following equations for $\ell_1$:
	\begin{equation*}
		\ell_1^2=B^2-\mathrm{Re}\,E_\pm,\quad
		\ell_1=-\frac{\mathrm{Im}\,E_\pm}{2B},
	\end{equation*}
	and thus we have items (i) and (ii) in Case $\widetilde{\mathrm{III}}$.
	\medskip

	\textbf{Proof of item (iii).}
	Combining \eqref{z1z2}, \eqref{ch-a1}, and \eqref{ch-p-a1}, we obtain 
	item (iii) of the proposition.
\end{proof}

\begin{corollary}[Reflectionless case]
	\label{rlc}
	Consider $b(k)=0$.
	Then \eqref{E+-} and \eqref{E12} yield that 	
	$E_\pm=\pm\I\frac{AB}{2}$.
	Since $\mathrm{Im}\,E_+>0$ and $\mathrm{Im}\,E_-<0$, we conclude from 
	item (i) of Proposition \ref{tza1}
	that only $E_-$ can be an admissible constant,
	and parameters $A$ and $B$ satisfy the following constraints:
	\begin{equation*}
		\mbox{Case $\widetilde{\mathrm{I}}$:}\quad
		0<B<\frac{A}{4},\quad
		\mbox{Case $\widetilde{\mathrm{II}}$:}\quad
		B>\frac{A}{4},\quad
		\mbox{Case $\widetilde{\mathrm{III}}$:}\quad
		B=\frac{A}{4}.
	\end{equation*}
	Moreover, item (ii) of Proposition \ref{tza1} implies that zeros of $a_1(k)$ read as follows:
	\begin{subequations}\label{zrlc}
	\begin{align}
		\label{zrlca}
		&\mbox{Case $\widetilde{\mathrm{I}}$:}\quad
		k_j=\frac{1}{4}\left(A+(-1)^j\sqrt{A^2-16B^2}\right),
		\quad j=1,2,\quad
		0<B<\frac{A}{4},\\
		\label{zrlcb}
		&\mbox{Case $\widetilde{\mathrm{II}}$:}\quad
		p_1=\frac{1}{4}\left(-\sqrt{16B^2-A^2}+\I A\right),
		\quad B>\frac{A}{4},\\
		\label{zrlcc}
		&\mbox{Case $\widetilde{\mathrm{III}}$:}\quad
		\ell_1=\frac{A}{4},\quad B=\frac{A}{4}.
	\end{align}
\end{subequations}
	Notice that these zeros of $a_1(k)$ are precisely the same as in the case of the pure oscillating step initial data, see Proposition \ref{a_1s} above.
\end{corollary}

\subsection{Basic Riemann-Hilbert problem}\label{BRHp}

Now we are at the position to formulate the basic Riemann-Hilbert problem associated to the Cauchy problem \eqref{nmkdvs}--\eqref{bcs}.
Consider the following sectionally meromorphic $2\times2$ matrix-valued function:
\begin{equation}
	\label{DM}
	M(x,t,k)=
	\left\{
	\begin{array}{lcl}
		\left(\frac{\Psi_1^{(1)}(x,t,k)}{a_{1}(k)},
		\Psi_2^{(2)}(x,t,k)\right),\quad k\in\C^+,\\
		\left(\Psi_2^{(1)}(x,t,k),\frac{\Psi_1^{(2)}(x,t,k)}
		{a_{2}(k)}\right),\quad k\in\C^-.\\
	\end{array}
	\right.
\end{equation}
Using the scattering relation \eqref{S} (recall also \eqref{Phi}), we conclude that 
$M(x,t,k)$ satisfies the following jump condition on
$\R\setminus\{-B,B\}$
(the real line is oriented from $-\infty$ to $\infty$):
\begin{equation}
	\label{Jj}
	M_+(x,t,k)=M_-(x,t,k)J(x,t,k),\quad k\in\R\setminus\{-B,B\},
\end{equation}
with
\begin{equation}\label{jump}
	J(x,t,k)=
	\begin{pmatrix}
		1+r_{1}(k)r_{2}(k)& r_{2}(k)e^{-2\I kx-8\I k^3t}\\
		r_1(k)e^{2\I kx+8\I k^3t}& 1
	\end{pmatrix},\quad
	k\in\R\setminus\{-B,B\}.
\end{equation}
Here the functions $r_j(k)$, $j=1,2$, referred as the reflection coefficients, are defined as follows:
\begin{equation}\label{r12}
	r_1(k)=\frac{b(k)}{a_1(k)},
	\quad k\in\R,\quad
	r_2(k)=\frac{b(k)}{a_2(k)},
	\quad k\in\R\setminus\{-B,B\}.
\end{equation}
Combining item (v) of Proposition \ref{aj-pr} and \eqref{r12},
we conclude that $r_j(k)$, $j=1,2$, have the following behavior as $k\to\pm B$:
\begin{equation*}
	\begin{split}
		&r_1(k)=-\I\frac{4}{A}(k\mp B)+O\left((k\mp B)^2\right),
		\quad k\to\pm B,\\
		&r_2(k)=-\I\frac{A}{4(k\mp B)}+O(1),\quad k\to\pm B,
	\end{split}
\end{equation*}
in every Case I, II, and III, as well as $\widetilde{\mathrm{I}}$, $\widetilde{\mathrm{II}}$, and $\widetilde{\mathrm{III}}$.
Then, item (iii) of Proposition \ref{aj-pr} implies the following symmetries of $r_j(k)$, $j=1,2$:
\begin{equation}
	\label{rjs}
	r_1(k)=\bar{r}_1(-k),\quad k\in\R,\quad
	r_2(k)=\bar{r}_2(-k),\quad k\in\R\setminus\{-B,B\},
\end{equation}
while the determinant relation given in \eqref{detR}
yields that
\begin{equation}
	\label{r-ja-j}
	1+r_1(k)r_2(k)=\frac{1}{a_1(k)a_2(k)},
	\quad k\in\R.
\end{equation}

Combining \eqref{DM} and Proposition \ref{pPsi}, items (ii)--(iii), we obtain the following  normalization condition at infinity of $M$:
\begin{equation}\label{Norm}
	M(x,t,k)\to I,\quad k\to\infty.
\end{equation}

Observe that $M(x,t,k)$ is not bounded at $k=\pm B$.
In the next proposition, we establish the behavior of $M(x,t,k)$
at these points.
\begin{proposition}[Singularities at $k=\pm B$]\label{Bcon}
	$M(x,t,k)$ satisfies the following conditions 
	as $k$ converges to $\pm B$ from the upper and lower complex half-planes:
	\begin{description}[font=\normalfont, itemindent=-\parindent]
		\item[\textit{Cases} I--III]
		\begin{equation}\label{pmB}
			\lim\limits_{k\to\pm B+\I0}
			M(x,t,k)\begin{pmatrix}
				\frac{1}{k\mp B}&0\\
				0& k\mp B
			\end{pmatrix}=
			M(x,t,\pm B-\I0)\begin{pmatrix}
				0&-\I\frac{A}{4}e^{\mp2\I Bx\mp8\I B^3t}\\
				-\I\frac{4}{A}e^{\pm2\I Bx\pm8\I B^3t}&0
			\end{pmatrix};
		\end{equation}
		
		\item[\textit{Cases} $\widetilde{\mathrm{I}}$--$\widetilde{\mathrm{III}}$]
		\begin{equation}\label{tpmB}
			\begin{split}
				&\lim\limits_{k\to\pm B+\I0}
				M(x,t,k)\begin{pmatrix}
					1&0\\
					0& k\mp B
				\end{pmatrix}=
				\left(\lim\limits_{k\to\pm B-\I0}
				M(x,t,k)\begin{pmatrix}
					1&0\\
					0& k\mp B
				\end{pmatrix}\right)
				Q_\pm(x,t),\\
				&\det\left(\lim\limits_{k\to\pm B+\I0}
				M(x,t,k)\begin{pmatrix}
					1&0\\
					0& k\mp B
				\end{pmatrix}\right)=0,
			\end{split}
		\end{equation}
		where the $2\times 2$ matrix $Q_\pm$ has the following form:
		\begin{equation*}
			Q_\pm(x,t)=-\I\frac{A}{4}\begin{pmatrix}
				0&e^{\mp2\I Bx\mp8\I B^3t}\\
				\frac{\left(a_2^\prime(\pm B)\right)^2}
				{1-b^2(\pm B)}e^{\pm2\I Bx\pm8\I B^3t}&0
			\end{pmatrix}.
		\end{equation*}
	\end{description}
\end{proposition}
\begin{proof}
Proposition \ref{pPsi}, item (vi), and \eqref{a_1+-B} imply that $M(x,t,k)$
admits the following expansion in Cases I--III
(cf.\,\,\cite[Equation (2.48)]{RS21-DE}, \cite[Proposition 4]{XF23}, and \cite[Section 2.4]{RST25}):
\begin{subequations}
\label{zprB}
\begin{equation}
\begin{split}
	&M(x,t,k)
	\begin{pmatrix}
		\frac{1}{k- B}& 0\\
		0& k- B
	\end{pmatrix}=
	\begin{pmatrix}
		\frac{16v_1(x,t)}{A^2a_2(B)}& v_2(-x,-t)\\
		\frac{16v_2(x,t)}{A^2a_2(B)}& v_1(-x,-t)
	\end{pmatrix}+O(k-B),\quad k\to B,\quad k\in\C^+,\\
	&M(x,t,k)=\frac{4\I}{A}
	\begin{pmatrix}
		e^{2\I Bx+8\I B^3t}v_2(-x,-t)&
		e^{-2\I Bx-8\I B^3t}\frac{v_1(x,t)}{a_2(B)}\\
		e^{2\I Bx+8\I B^3t}v_1(-x,-t)&
		e^{-2\I Bx-8\I B^3t}\frac{v_2(x,t)}{a_2(B)}
	\end{pmatrix}\\
	&\qquad\qquad\quad\,\,+O(k-B),\quad k\to B,\quad k\in\C^-,
\end{split}
\end{equation}
and
\begin{equation}
	\begin{split}
		&M(x,t,k)
		\begin{pmatrix}
			\frac{1}{k+ B}& 0\\
			0& k+ B
		\end{pmatrix}=
		\begin{pmatrix}
			-\frac{16\bar{v}_1(x,t)}{A^2a_2(-B)}& -\bar{v}_2(-x,-t)\\
			-\frac{16\bar{v}_2(x,t)}{A^2a_2(-B)}& -\bar{v}_1(-x,-t)
		\end{pmatrix}+O(k+B),\quad k\to -B,\quad k\in\C^+,\\
		&M(x,t,k)=-\frac{4\I}{A}
		\begin{pmatrix}
			e^{-2\I Bx-8\I B^3t}\bar{v}_2(-x,-t)&
			e^{2\I Bx+8\I B^3t}\frac{\bar{v}_1(x,t)}{a_2(-B)}\\
			e^{-2\I Bx-8\I B^3t}\bar{v}_1(-x,-t)&
			e^{2\I Bx+8\I B^3t}\frac{\bar{v}_2(x,t)}{a_2(-B)}
		\end{pmatrix}\\
		&\qquad\qquad\quad\,\,+O(k+B),\quad k\to -B,\quad k\in\C^-.
	\end{split}
\end{equation}
\end{subequations}
Applying again item (vi) of Proposition \ref{pPsi} and \eqref{b(B)}, we obtain the following behavior of $M$ as $k\to\pm B$ in 
	Cases $\widetilde{\mathrm{I}}$--$\widetilde{\mathrm{III}}$:
	\begin{subequations}
		\label{IzprB}
	\begin{equation}
		\begin{split}
			&M(x,t,k)
			\begin{pmatrix}
				1& 0\\
				0& k- B
			\end{pmatrix}=
			\begin{pmatrix}
				\frac{a_2^\prime(B)v_1(x,t)}{1-b^2(B)}& v_2(-x,-t)\\
				\frac{a_2^\prime(B)v_2(x,t)}{1-b^2(B)}& v_1(-x,-t)
			\end{pmatrix}+O(k-B),\quad k\to B,\quad k\in\C^+,\\
			&M(x,t,k)\begin{pmatrix}
				1& 0\\
				0& k- B
			\end{pmatrix}
			=\frac{4\I}{A}
			\begin{pmatrix}
				e^{2\I Bx+8\I B^3t}v_2(-x,-t)&
				e^{-2\I Bx-8\I B^3t}\frac{v_1(x,t)}{a_2^{\prime}(B)}\\
				e^{2\I Bx+8\I B^3t}v_1(-x,-t)&
				e^{-2\I Bx-8\I B^3t}\frac{v_2(x,t)}{a_2^{\prime}(B)}
			\end{pmatrix}\\
			&\qquad\qquad\qquad\qquad\qquad\quad\,
			+O(k-B),\quad k\to B,\quad k\in\C^-,
		\end{split}
	\end{equation}
	and
	\begin{equation}
		\begin{split}
			&M(x,t,k)
			\begin{pmatrix}
				1& 0\\
				0& k+B
			\end{pmatrix}=
			-\begin{pmatrix}
				\frac{a_2^\prime(-B)\bar{v}_1(x,t)}{1-b^2(-B)}& 
				\bar{v}_2(-x,-t)\\
				\frac{a_2^\prime(-B)\bar{v}_2(x,t)}{1-b^2(-B)}&
				\bar{v}_1(-x,-t)
			\end{pmatrix}+O(k+B),\quad k\to -B,\quad k\in\C^+,\\
			&M(x,t,k)\begin{pmatrix}
				1& 0\\
				0& k+B
			\end{pmatrix}
			=-\frac{4\I}{A}
			\begin{pmatrix}
				e^{-2\I Bx-8\I B^3t}\bar{v}_2(-x,-t)&
				e^{2\I Bx+8\I B^3t}\frac{\bar{v}_1(x,t)}{a_2^{\prime}(-B)}\\
				e^{-2\I Bx-8\I B^3t}\bar{v}_1(-x,-t)&
				e^{2\I Bx+8\I B^3t}\frac{\bar{v}_2(x,t)}{a_2^{\prime}(-B)}
			\end{pmatrix}\\
			&\qquad\qquad\qquad\qquad\qquad\quad\,
			+O(k+B),\quad k\to -B,\quad k\in\C^-,
		\end{split}
	\end{equation}
	\end{subequations}
	where (recall \eqref{a2cl})
	\begin{equation}\label{dettB}
		v_1(x,t)v_1(-x,-t)-v_2(x,t)v_2(-x,-t)=0,\quad\mbox{for all }\, x,t\in\R.
	\end{equation}

Then \eqref{pmB} and \eqref{tpmB} follow from \eqref{zprB} and \eqref{IzprB}--\eqref{dettB}, respectively.
\end{proof}

Recalling assumptions on zeros of $a_1(k)$ in Cases I--III and 
$\widetilde{\mathrm{I}}$--$\widetilde{\mathrm{III}}$, as well as the analytical properties of the columns $\Psi_i^{(j)}$, $i,j=1,2$, described in items (ii)--(iii) of Proposition \ref{pPsi},
we conclude from \eqref{DM} that the first column of
$M(x,t,k)$ has poles at zeros of $a_1(k)$ for $k\in\C^+$.
The behavior of $M$ near these poles is described in the next proposition.

\begin{proposition}[Residue conditions]
\label{Rcon}
Define the norming constants $\gamma_j$, $j=1,2$, $\eta_1$, and $\nu_1$
in terms of the initial data $u_0(x)$ as follows
(recall \eqref{Psi1}, \eqref{Psi2}, and \eqref{sda}):
\begin{subequations}
	\label{d-nct}
\begin{align}
	&\mbox{Cases $\mathrm{I}$, $\widetilde{\mathrm{I}}$:}\quad
	\Psi_1^{(1)}(0,0,\I k_j)=\gamma_j\Psi_2^{(2)}(0,0,\I k_j),
	\quad j=1,2,\\
	&\mbox{Cases $\mathrm{II}$, $\widetilde{\mathrm{II}}$:}\quad	\Psi_1^{(1)}(0,0,p_1)=\eta_1\Psi_2^{(2)}(0,0,p_1),\\
	\label{d-nctc}
	&\mbox{Cases $\mathrm{III}$, $\widetilde{\mathrm{III}}$:}\quad	\Psi_1^{(1)}(0,0,\I\ell_1)=\nu_1\Psi_2^{(2)}(0,0,\I\ell_1).
\end{align}
\end{subequations}
Then we have that
\begin{equation}\label{nct}
	\gamma_j=\pm1,\,\,j=1,2,\quad\eta_1=\pm1,\quad
	\nu_1=\pm1,
\end{equation}
and the first column of the $2\times2$ matrix $M(x,t,k)$, given in \eqref{DM}, has the following poles in $\C^+$:
\begin{description}[font=\normalfont, itemindent=-\parindent]
	\item[\textit{Cases} I, $\widetilde{\mathrm{I}}$]
	$M^{(1)}$ has simple poles at $k=\I k_j$, $j=1,2$, which satisfy the following residue conditions:
	\begin{equation}\label{rskj}
		\underset{k=\I k_j}{\operatorname{Res}} M^{(1)}(x,t,k)=
		\frac{\gamma_j}{a_1^\prime(\I k_j)}
		e^{-2k_jx+8k_j^3t}M^{(2)}(x,t,\I k_j),
		\quad j=1,2;
	\end{equation}
	
	\item[\textit{Cases} II, $\widetilde{\mathrm{II}}$]
	$M^{(1)}$ has simple poles at $k=p_1$ and $k=-\bar{p}_1$, which satisfy the following residue conditions:
	\begin{equation}\label{rsp1}
	\begin{split}
		&\underset{k=p_1}{\operatorname{Res}} M^{(1)}(x,t,k)=
		\frac{\eta_1}{a_1^\prime(p_1)}
		e^{2\I p_1x+8\I p_1^3t}M^{(2)}(x,t,p_1),\\
		&\underset{k=-\bar{p}_1}{\operatorname{Res}} M^{(1)}(x,t,k)=
		\frac{\eta_1}{a_1^\prime(-\bar{p}_1)}
		e^{-2\I\bar{p}_1x-8\I\bar{p}_1^3t}
		M^{(2)}(x,t,-\bar{p}_1);
	\end{split}
	\end{equation}
	
	\item[\textit{Cases} III, $\widetilde{\mathrm{III}}$]
	$M^{(1)}$ has a double pole at $k=\I\ell_1$, which satisfies the following conditions:
	\begin{subequations}\label{rsl1}
	\begin{align}
		\label{rsl1-a}
		&\underset{k=\I\ell_1}{\operatorname{Res}}(k-\I\ell_1)M^{(1)}(x,t,k)=
		\frac{2\nu_1}{a_1^{\prime\prime}(\I\ell_1)}
		e^{-2\ell_1x+8\ell_1^3t}M^{(2)}(x,t,\I\ell_1),\\
		\nonumber
		&\underset{k=\I\ell_1}{\operatorname{Res}} M^{(1)}(x,t,k)=
		\frac{2\nu_1}{a_1^{\prime\prime}(\I\ell_1)}
		e^{-2\ell_1x+8\ell_1^3t}\left(
		\partial_kM^{(2)}(x,t,\I\ell_1)\right.\\
		\label{rsl1-b}
		&\left.\qquad\qquad\qquad\qquad\qquad\qquad
		\qquad\qquad\quad
		+\left(2\I(x-12\ell_1^2t)
		-\frac{a_1^{\prime\prime\prime}(\I\ell_1)}
		{3a_1^{\prime\prime}(\I\ell_1)}\right)
		M^{(2)}(x,t,\I\ell_1)
		\right).
	\end{align}
	\end{subequations}
\end{description}
\end{proposition}
\begin{proof}
	\textbf{Step 1.}
	Using the first symmetry relation in \eqref{Psi-sym}, we obtain 
	that
	\begin{equation}\label{PsC+}
		\left(\Psi_1\right)_{11}(0,0,k)=\left(\Psi_2\right)_{22}(0,0,k),
		\quad\mbox{and}\quad
		\left(\Psi_1\right)_{21}(0,0,k)=\left(\Psi_2\right)_{12}(0,0,k),
		\quad k\in\C^+,
	\end{equation}
	which, together with \eqref{d-nct}, imply
	\eqref{nct}.
	Also, notice that in Cases II and $\widetilde{\mathrm{II}}$ 
	there exists a constant $\hat{\eta}_1$ such that
	\begin{equation}\label{Psi-et}
		\Psi_1^{(1)}(0,0,-\bar{p}_1)=
		\hat{\eta}_1\Psi_2^{(2)}(0,0,-\bar{p}_1).
	\end{equation}
	Then the second
	symmetry relation in \eqref{Psi-sym} implies that
	$\hat{\eta}_1=\bar{\eta}_1^{-1}$ and therefore (recall that $\eta_1=\pm1$)
	\begin{equation}\label{et-h}
		\hat{\eta}_1=\eta_1.
	\end{equation}

	\textbf{Step 2.}
	Using that the columns $\Phi_i^{(j)}(x,t,k)$, $i,j=1,2$,
	solve the Lax pair \eqref{LPnmkdv}, and that 
	$\Phi_i^{(j)}(0,0,k)=\Psi_i^{(j)}(0,0,k)$, $i,j=1,2$, recall \eqref{Phi},
	we arrive at the following relations 
	(see \eqref{d-nct}, \eqref{Psi-et}, and \eqref{et-h}):
	\begin{subequations}
	\label{P-nct}
	\begin{align}
		\label{P-ncta}
		&\mbox{Cases $\mathrm{I}$, $\widetilde{\mathrm{I}}$:}\quad
		\Phi_1^{(1)}(x,t,\I k_j)=\gamma_j\Phi_2^{(2)}(x,t,\I k_j),
		\quad j=1,2,\\
		\label{P-nctb}
		&\mbox{Cases $\mathrm{II}$, $\widetilde{\mathrm{II}}$:}\quad
		\Phi_1^{(1)}(x,t,p_1)=\eta_1\Phi_2^{(2)}(x,t,p_1),
		\quad \Phi_1^{(1)}(x,t,-\bar{p}_1)=\eta_1\Phi_2^{(2)}(x,t,-\bar{p}_1),\\
		\label{P-nctc}
		&\mbox{Cases $\mathrm{III}$, $\widetilde{\mathrm{III}}$:}\quad
		\Phi_1^{(1)}(x,t,\I\ell_1)=\nu_1\Phi_2^{(2)}(x,t,\I\ell_1),
	\end{align}
	\end{subequations}
	for all $x,t\in\R$.
	Relations \eqref{P-ncta}, \eqref{P-nctb}, and \eqref{P-nctc},
	together with \eqref{DM}, imply \eqref{rskj}, \eqref{rsp1}, and \eqref{rsl1-a}, respectively.
	\medskip
	
	\textbf{Step 3.}
	It remains to establish \eqref{rsl1-b} for Cases $\mathrm{III}$ and $\widetilde{\mathrm{III}}$.
	Recalling that $a_1(k)$ has a double zero at $k=\I\ell_1$,
	we have the following Laurent expansion of $M^{(1)}(x,t,k)$ (recall \eqref{DM}):
	\begin{equation}\label{M1L}
	\begin{split}
		M^{(1)}(x,t,k)=\,&2\frac{\Psi^{(1)}_1(x,t,\I\ell_1)}
		{a_1^{\prime\prime}(\I\ell_1)}
		(k-\I\ell_1)^{-2}\\
		&+2\left(\frac{\partial_k\Psi_1^{(1)}(x,t,\I\ell_1)}
		{a_1^{\prime\prime}(\I\ell_1)}
		-\frac{\Psi_1^{(1)}(x,t,\I\ell_1)a_1^{\prime\prime\prime}(\I\ell_1)}
		{3\left(a_1^{\prime\prime}\right)^2(\I\ell_1)}
		\right)(k-\I\ell_1)^{-1}+O(1),\quad k\to\I\ell_1.
	\end{split}
	\end{equation}
	
	Let us prove that (cf.\,\,\eqref{P-nctc})
	\begin{equation}\label{pkPhil1}
		\partial_k\Phi_1^{(1)}(x,t,\I\ell_1)
		=\nu_1\partial_k\Phi_2^{(2)}(x,t,\I\ell_1).
	\end{equation}
	Observe that 
	\eqref{S}, together with the Cramer's rule, yields the following formula for $a_1(k)$ (cf.\,\,\eqref{sda}):
	\begin{equation}\label{a1de}
		a_1(k)=\det\left(
		\Phi_1^{(1)}(x,t,k),\Phi_2^{(2)}(x,t,k)
		\right),
		\quad x,t\in\R,
		\quad k\in\overline{\C^+}
		\setminus\{-B,B\}.
	\end{equation}
	Using \eqref{P-nctc} and that $a_1^\prime(\I\ell_1)=0$,
	we obtain from \eqref{a1de} the following equation:
	\begin{equation}\label{de0}
		\det\left(
		\partial_k\left(\Phi_1^{(1)}
		-\nu_1\Phi_2^{(2)}\right)(x,t,\I\ell_1),
		\Phi_2^{(2)}(x,t,\I\ell_1)
		\right)=0,
		\quad x,t\in\R.
	\end{equation}
	Since $\left(\Phi_1^{(1)}-\nu_1\Phi_2^{(2)}\right)(x,t,\I\ell_1)=0$,
	we have that both 
	$\partial_k\left(\Phi_1^{(1)}-\nu_1\Phi_2^{(2)}\right)(x,t,\I\ell_1)$
	and $\Phi_2^{(2)}(x,t,\I\ell_1)$ satisfy \eqref{LPnmkdv}.
	Therefore, \eqref{de0} implies that there exists a constant $\rho_1\in\C$ such that
	\begin{equation}\label{rho1}
		\partial_k\left(\Phi_1^{(1)}
		-\nu_1\Phi_2^{(2)}\right)(x,t,\I\ell_1)
		=\rho_1\Phi_2^{(2)}(x,t,\I\ell_1),\quad x,t\in\R.
	\end{equation}
	Combining \eqref{Phi}, \eqref{PsC+}, and \eqref{rho1}, we obtain the following system of equations:
	\begin{equation}\label{spk}
	\begin{split}
		&\partial_k\left(\left(\Psi_2\right)_{22}
		-\nu_1\left(\Psi_2\right)_{12}\right)(0,0,\I\ell_1)=
		\rho_1\left(\Psi_2\right)_{12}(0,0,\I\ell_1),\\
		&\partial_k\left(\left(\Psi_2\right)_{12}
		-\nu_1\left(\Psi_2\right)_{22}\right)(0,0,\I\ell_1)=
		\rho_1\left(\Psi_2\right)_{22}(0,0,\I\ell_1).
	\end{split}
	\end{equation}
	Recalling that $\nu_1^2=1$, we have from \eqref{spk} that
	\begin{equation*}
		\rho_1\left(\nu_1\left(\Psi_2\right)_{12}
		+\left(\Psi_2\right)_{22}\right)(0,0,\I\ell_1)=0,
	\end{equation*}
	which, together with the relation (see \eqref{d-nctc} and the second equation in \eqref{PsC+})
	$$\left(\Psi_2\right)_{12}(0,0,\I\ell_1)
	=\nu_1\left(\Psi_2\right)_{22}(0,0,\I\ell_1),$$
	yields that 
	\begin{equation}\label{rho01}
		\rho_1=0.
	\end{equation}
	Combining \eqref{rho1}, and \eqref{rho01}, we obtain \eqref{pkPhil1}.
	Finally, \eqref{rsl1-b} follows from \eqref{M1L}, \eqref{pkPhil1},
	\eqref{P-nctc}, \eqref{Phi}, and \eqref{DM}.
\end{proof}

Given $M(x,t,k)$, the solution $u(x,t)$ of the original Cauchy problem 
\eqref{nmkdvs}--\eqref{bcs} can be recovered via the large $k$ expansion of $M(x,t,k)$ as follows:
\begin{equation}\label{msol}
\begin{split}
	&u(x,t)=2\I\lim_{k\to\infty}
	kM_{12}(x,t,k),\\
	&u(-x,-t)=-2\I\lim_{k\to\infty}
	kM_{21}(x,t,k).
\end{split}
\end{equation}

Summarizing, we arrive at the following theorem.
\begin{theorem}[Basic Riemann-Hilbert problem]\label{TBRHp}
Consider the initial data $u_0(x)$ such that 
$$
xu_0(x)\in L^1(-\infty,a),\quad 
\left(
u_0(x)-A\cos2Bx
\right)\in L^1(a,\infty),
$$ with respect to the spatial variable $x$, for all fixed $a\in\R$.
Assume that the spectral functions $a_j(k)$, $j=1,2$, and $b(k)$, defined by \eqref{sd}, satisfy conditions described in either Case $\mathrm{I}$, $\mathrm{II}$, or $\mathrm{III}$, or  
Case $\widetilde{\mathrm{I}}$,
$\widetilde{\mathrm{II}}$, or $\widetilde{\mathrm{III}}$. 
Consider the basic Riemann-Hilbert problem, which consists in
finding the $2\times2$ matrix valued function
$M(x,t,k)$ satisfying the following conditions:
\begin{enumerate}
	\item multiplicative jump condition \eqref{Jj};
	\item normalization condition \eqref{Norm};
	\item residue conditions given in Proposition \ref{Rcon};
	\item singularity conditions \eqref{pmB}--\eqref{tpmB} at $k=\pm B$.
\end{enumerate}

Then, the solution $u(x,t)$ of the Cauchy problem \eqref{nmkdvs} 
can be reconstructed in terms of
the solution of the basic Riemann-Hilbert problem by \eqref{msol}.
Moreover, the functions $a_j(k)$, $j=1,2$,
can be calculated in terms of $b(k)$
according to Propositions \ref{za1a2} and \ref{tza1}. 
\end{theorem}
\begin{remark}[Uniqueness]
	Observe that the solution of the basic Riemann-Hilbert problem
	is unique, if it exists.
	Indeed, taking into account jump condition \eqref{Jj},
	residue conditions \eqref{rskj}--\eqref{rsl1},
	and singularity conditions \eqref{pmB}--\eqref{tpmB}, we conclude that
	$\det M(x,t,k)$ is an entire function in every 
	Case $\mathrm{I}$--$\mathrm{III}$ and 
	$\widetilde{\mathrm{I}}$--$\widetilde{\mathrm{III}}$.
	Since $\det M(x,t,k)\to 1$, see \eqref{Norm}, we have that $\det M(x,t,k)=1$ for all $x,t\in\R$ and $k\in\C$ by the Liouville theorem.
	Therefore, if $M(x,t,k)$ and $\tilde{M}(x,t,k)$ are two solutions of the basic Riemann-Hilbert problem, we can consider the product 
	$\left(M\tilde{M}^{-1}\right)(x,t,k)$, which is continuous along $\R$ and is bounded at every singular point (see \eqref{rskj}--\eqref{rsl1} and \eqref{pmB}--\eqref{tpmB}).
	Finally, using normalization condition \eqref{Norm} and the Liouville theorem, 
	we conclude that 
	$
	\left(M\tilde{M}^{-1}\right)(x,t,k)=I,
	$
	for all $x,t\in\R$ and $k\in\C$.
\end{remark}

Notice that the solution of the Basic Riemann-Hilbert problem, if it exists, satisfies the symmetry conditions described in the next proposition.
\begin{proposition}[Symmetries]
	Assume that $M(x,t,k)$ is a solution of the basic Riemann-Hilbert problem.
	Then it satisfies the following symmetry conditions
	(cf.\,\,\cite[Equation (62)]{XF23}, \cite[Proposition 4]{RS21-DE}):
	\begin{equation}\label{M-symm}
		M(x,t,k)=
		\begin{cases}
			\sigma_1\overline{M}(-x,-t,-\bar{k})\sigma_1^{-1}
			\begin{pmatrix}
				\frac{1}{a_1(k)}&0\\
				0&a_1(k)
			\end{pmatrix},\quad k\in\C^+,\\
			\sigma_1\overline{M}(-x,-t,-\bar{k})\sigma_1^{-1}
			\begin{pmatrix}
				a_2(k)&0\\
				0&\frac{1}{a_2(k)}
			\end{pmatrix},\quad k\in\C^-,
		\end{cases}
	\end{equation}
	and
	\begin{equation}
		\label{M-symm-1}
		M(x,t,k)=
		\begin{cases}
			\sigma_1M(-x,-t,k)\sigma_1^{-1}
			\begin{pmatrix}
				\frac{1}{a_1(k)}&0\\
				0&a_1(k)
			\end{pmatrix},\quad k\in\C^+,\\
			\sigma_1M(-x,-t,k)\sigma_1^{-1}
			\begin{pmatrix}
				a_2(k)&0\\
				0&\frac{1}{a_2(k)}
			\end{pmatrix},\quad k\in\C^-.
		\end{cases}
	\end{equation}
\end{proposition}
\begin{proof}
	Combining \eqref{r12}, \eqref{rjs}, and \eqref{r-ja-j}, we conclude that
	the jump matrix $J(x,t,k)$ satisfies the following symmetry conditions:
	\begin{equation}\label{J-symm}
		\sigma_1\overline{J}(-x,-t,-k)\sigma_1^{-1}
		=\begin{pmatrix}
			a_2(k)&0\\
			0&\frac{1}{a_2(k)}
		\end{pmatrix}
		J(x,t,k)
		\begin{pmatrix}
			a_1(k)&0\\
			0&\frac{1}{a_1(k)}
		\end{pmatrix},\quad k\in\R\setminus\{-B,B\},
	\end{equation}
	and
	\begin{equation}\label{J-symm-1}
		\sigma_1J(-x,-t,k)\sigma_1^{-1}
		=\begin{pmatrix}
			a_2(k)&0\\
			0&\frac{1}{a_2(k)}
		\end{pmatrix}
		J(x,t,k)
		\begin{pmatrix}
			a_1(k)&0\\
			0&\frac{1}{a_1(k)}
		\end{pmatrix},\quad k\in\R\setminus\{-B,B\},
	\end{equation}
	which imply that the right-hand sides of \eqref{M-symm} and \eqref{M-symm-1} satisfy the jump condition \eqref{Jj}.
	Using the symmetry conditions given in item (iii) of Proposition \ref{aj-pr}, we conclude that 
	the right-hand sides of \eqref{M-symm} and \eqref{M-symm-1} satisfy the singularity conditions as prescribed by Propositions \ref{Bcon} and \ref{Rcon}.
	Finally, taking into account the uniqueness of the solution of the basic Riemann-Hilbert problem, we arrive at the relations \eqref{M-symm} and \eqref{M-symm-1}.
\end{proof}

\section{Two-soliton solutions}\label{twsol}

Here we consider Cases $\widetilde{\mathrm{I}}$--$\widetilde{\mathrm{III}}$
in the reflectionless case, i.e., with $b(k)=0$
(recall that \eqref{b+-B} yields that if $b(k)=0$, then $a_2(\pm B)=0$).
Recalling item (iii) of Proposition \ref{tza1} and Corollary \ref{rlc},
we obtain the following equations for the spectral functions $a_j(k)$, $j=1,2$:
\begin{equation}\label{a1rc}
	a_1(k)=\frac{(k-z_1)(k-z_2)}{k^2-B^2},\quad
	a_2(k)=\frac{k^2-B^2}{(k-z_1)(k-z_2)},
\end{equation}
where $(z_1,z_2)$ is given in \eqref{z1z2-t} with $k_j$, $j=1,2$, $p_1$, and $\ell_1$ given by \eqref{zrlc}.
Since $b(k)=0$, the basic Riemann-Hilbert problem for $M$
has no jump condition across $\R$ (see \eqref{jump} and \eqref{r12}),
and therefore the first column of $M$ is analytic in the neighborhoods of
$k=\pm B$, while the second column has simple poles at $k=\pm B$,
see \eqref{tpmB}.
These imply that
$M$ satisfies the following residue conditions at $k=\pm B$:
\begin{subequations}
	\label{res+-B}
\begin{align}
	\label{res+-Ba}
	&\underset{k=B}{\operatorname{Res}} M^{(2)}(x,t,k)=
	-\I\frac{A}{4}
	e^{-2\I Bx-8\I B^3t}M^{(1)}(x,t,B),\\
	\label{res+-Bb}
	&\underset{k=-B}{\operatorname{Res}} M^{(2)}(x,t,k)=
	-\I\frac{A}{4}
	e^{2\I Bx+8\I B^3t}M^{(1)}(x,t,-B).
\end{align}
\end{subequations}

Thus, the Riemann-Hilbert problem for $M(x,t,k)$ consists in finding a 
meromorphic $2\times2$ matrix-valued function, 
which satisfies residue conditions \eqref{res+-B} and \eqref{rskj}--\eqref{rsl1} (with $a_1(k)$ given by \eqref{a1rc}),
as well as the normalization condition \eqref{Norm}.
Such a Riemann-Hilbert problem can be solved explicitly, which is discussed in the lemma below in the case of the two pairs of simple zeros.
\begin{lemma}\label{L1nm}
	Consider the following Riemann-Hilbert problem 
	for $\tilde{M}(x,t,k)$
	\begin{subequations}\label{RHtM}
		\begin{align}
			\label{RHtMa}
			&\underset{k=w_j}{\operatorname{Res}} \tilde{M}^{(1)}(x,t,k)=
			c_j(x,t)\tilde{M}^{(2)}(x,t,w_j),\quad j=1,2,\\
			\label{RHtMb}
			&\underset{k=q_j}{\operatorname{Res}} \tilde{M}^{(2)}(x,t,k)=
			f_j(x,t)\tilde{M}^{(1)}(x,t,q_j),\quad j=1,2,\\
			& \tilde{M}(x,t,k)\to I,\quad k\to\infty,
		\end{align}
	\end{subequations}
	with some functions $c_j(x,t)$, $f_j(x,t)$, $j=1,2$, and 
	pairwise different complex numbers $w_j$, $q_j$, $j=1,2$
	(here we assume that $\tilde{M}^{(1)}$ and $\tilde{M}^{(2)}$ have simple poles at the considered points).
	Introduce the $2\times2$ matrix $N(x,t)$, whose elements have the following form:
	\begin{equation}\label{Nij}
		N_{ij}(x,t)=\frac{\left(\xi_j^T\zeta_i\right)(x,t)}{q_i-w_j},
		\quad i,j=1,2,
	\end{equation}
	where
	\begin{subequations}
		\label{xi-zeta}
	\begin{align}
		\label{xi-zeta-a}
		&\xi_j(x,t)=\begin{pmatrix}
			\frac{(-1)^j(w_2-w_1)}{(w_j-q_1)(w_j-q_2)}c_j(x,t)
			\\1
		\end{pmatrix}\quad j=1,2,\\
		\label{xi-zeta-b}
		&\zeta_j(x,t)=\begin{pmatrix}
			\frac{(-1)^j(q_1-q_2)}{(q_j-w_1)(q_j-w_2)}f_j(x,t)
			\\1
		\end{pmatrix}
		\quad j=1,2.
	\end{align}
	\end{subequations}
	
	Then problem \eqref{RHtM} can be solved explicitly 
	by \eqref{Mtrep}, \eqref{Ajz}, and \eqref{zjsys},
	for all values of the parameters $x,t$ such that $\det N(x,t)\neq 0$.
	Moreover, the large $k$ limit of the 
	$(1,2)$ and $(2,1)$ elements of $\tilde{M}$
	can be found as follows:
	\begin{equation}\label{M12sol}
		\lim_{k\to\infty}k\tilde{M}_{12}(x,t,k)=
		\frac{\det N_1(x,t)}{\det N(x,t)},
		\quad \det N(x,t)\neq0,
	\end{equation}
	and
	\begin{equation}\label{M21sol}
		\lim_{k\to\infty}k\tilde{M}_{21}(x,t,k)=
		\frac{\det N_2(x,t)}{\det N(x,t)},
		\quad \det N(x,t)\neq0,
	\end{equation}
	where ($\xi_{j,1}$ and $\zeta_{j,1}$ denote 
	the $(1,1)$ element of $\xi_j$ and $\zeta_j$, $j=1,2$, respectively)
	\begin{equation}\label{N-1-xi-zeta}
		N_1(x,t)=
		\left(
		\begin{array}{@{}c|c@{}}
			N(x,t) &
			\begin{matrix}
				\zeta_{1,1}(x,t)\\
				\zeta_{2,1}(x,t)
			\end{matrix}\\
			\hline
			\begin{matrix}
				1&1
			\end{matrix}
			&0
		\end{array}
		\right),
	\end{equation}
	and
	\begin{equation*}
		N_2(x,t)=
		\left(
		\begin{array}{@{}c|c@{}}
			N(x,t) &
			\begin{matrix}
				1\\
				1
			\end{matrix}\\
			\hline
			\begin{matrix}
				\xi_{1,1}(x,t)&\xi_{2,1}(x,t)
			\end{matrix}
			&0
		\end{array}
		\right).
	\end{equation*}
\end{lemma}

\begin{proof}
	The proof closely follows the methodology presented in
	\cite[Chapter II, \S 5]{FT87}.
	Let us seek the solution $\tilde{M}(x,t,k)$ of the Riemann-Hilbert problem \eqref{RHtM} in the following form:
	\begin{equation}\label{Mtrep}
		\tilde{M}(x,t,k)=\left(I+\frac{A_1(x,t)}{k-w_1}+
		\frac{A_2(x,t)}{k-w_2}\right)
		\begin{pmatrix}
			1&0\\
			0&\frac{(k-w_1)(k-w_2)}{(k-q_1)(k-q_2)}
		\end{pmatrix},
	\end{equation}
	where the $2\times2$ matrices $A_j(x,t)$, $j=1,2$, are yet to be determined.
	Representation \eqref{Mtrep} yields that
	\begin{equation}\label{A-col}
	\begin{split}
		&\underset{k=w_j}{\operatorname{Res}} \tilde{M}^{(1)}(x,t,k)
		=A_j^{(1)}(x,t),\quad j=1,2,\\
		&\tilde{M}^{(2)}(x,t,w_j)=\frac{(-1)^j(w_2-w_1)}{(w_j-q_1)(w_j-q_2)}
		A_j^{(2)}(x,t),\quad j=1,2.
	\end{split}
	\end{equation}
	Combining \eqref{RHtMa} and \eqref{A-col},
	we conclude that
	\begin{equation}\label{Ajz}
		A_j(x,t)=\left(z_j\xi_j^T\right)(x,t),\quad j=1,2,
	\end{equation}
	with some $2$-vector $z_j(x,t)$ and $\xi_j(x,t)$ given by \eqref{xi-zeta-a}.
	
	Equations \eqref{Mtrep} and \eqref{Ajz} yield
	\begin{equation}\label{sys1}
		I+\frac{\left(z_1\xi_1^T\right)(x,t)}{k-w_1}
		+\frac{\left(z_2\xi_2^T\right)(x,t)}{k-w_2}
		=\left(\tilde{M}^{(1)}(x,t,k),
		\frac{(k-q_1)(k-q_2)}{(k-w_1)(k-w_2)}\tilde{M}^{(2)}(x,t,k)\right).
	\end{equation}
	Multiplying \eqref{sys1} by $\zeta_j$ on the right, see \eqref{xi-zeta-b},
	considering $k=q_j$, and employing \eqref{RHtMb}, we obtain the following system of equations for $z_j$:
	\begin{equation*}
		\zeta_j(x,t)+\frac{\left(\xi_1^T\zeta_j\right)(x,t)}{q_j-w_1}
		z_1(x,t)+\frac{\left(\xi_2^T\zeta_j\right)(x,t)}{q_j-w_2}z_2(x,t)=0,
		\quad j=1,2,
	\end{equation*}
	or, equivalently (recall \eqref{Nij}),
	\begin{equation}\label{zjsys}
		N(x,t)\begin{pmatrix}
			z_1^T\\z_2^T
		\end{pmatrix}(x,t)
		=-\begin{pmatrix}
			\zeta_1^T\\
			\zeta_2^T
		\end{pmatrix}(x,t).
	\end{equation}
	Equation \eqref{zjsys} uniquely determines $z_j(x,t)$, $j=1,2$, for all $x,t$ such that $\det N(x,t)\neq0$, and therefore we obtain the solution of \eqref{RHtM} through \eqref{Mtrep} and \eqref{Ajz}.
	Then, observing that (see \eqref{Mtrep}, \eqref{Ajz}, and \eqref{xi-zeta};
	here $z_j=\left(z_{j,1},z_{j,2}\right)^T$)
	\begin{equation*}
	\begin{split}
		&\lim_{k\to\infty}k\tilde{M}_{12}(x,t,k)=
		\left((A_1)_{12}+(A_2)_{12}\right)(x,t)=
		\left(z_{1,1}+z_{2,1}\right)(x,t),\\
		&\lim_{k\to\infty}k\tilde{M}_{21}(x,t,k)=
		\left((A_1)_{21}+(A_2)_{21}\right)(x,t)=
		\left(z_{1,2}\xi_{1,1}+z_{2,2}\xi_{2,1}\right)(x,t),
	\end{split}
	\end{equation*}
	and applying the Cramer's rule to \eqref{zjsys}, we arrive at \eqref{M12sol} and \eqref{M21sol}.
\end{proof}

Now let us examine the Riemann-Hilbert problem having two simple poles for the second column, and one second order pole for the first column.
\begin{lemma}\label{L2nm}
	Consider the following Riemann-Hilbert problem 
	for $\tilde{M}(x,t,k)$
	\begin{subequations}\label{dRHtM}
		\begin{align}
			\label{dRHtMa}
			&\underset{k=w_1}{\operatorname{Res}} \tilde{M}^{(1)}(x,t,k)=
			\tilde{c}_1(x,t)\partial_k\tilde{M}^{(2)}(x,t,w_1)
			+\tilde{c}_2(x,t)\tilde{M}^{(2)}(x,t,w_1),\\
			\label{dRHtMb}
			&\underset{k=w_1}{\operatorname{Res}}
			(k-w_1)\tilde{M}^{(1)}(x,t,k)=
			\tilde{c}_3(x,t)\tilde{M}^{(2)}(x,t,w_1),\\
			\label{dRHtMc}
			&\underset{k=q_j}{\operatorname{Res}} \tilde{M}^{(2)}(x,t,k)=
			\tilde{f}_j(x,t)\tilde{M}^{(1)}(x,t,q_j),\quad j=1,2,\\
			& \tilde{M}(x,t,k)\to I,\quad k\to\infty,
		\end{align}
	\end{subequations}
	with some functions $\tilde{c}_j(x,t)$, 
	$\tilde{f}_i(x,t)$, $j=1,2,3$, $i=1,2$, and 
	pairwise different complex numbers $w_1$, $q_j$, $j=1,2$
	(here we assume that $\tilde{M}^{(1)}$ has a double pole at $k=w_1$ and $\tilde{M}^{(2)}$ has simple poles at $k=q_j$, $j=1,2$).
	Introduce the $2\times2$ matrix $\tilde{N}(x,t)$,
	whose elements have the following form:
	\begin{equation}\label{tNij}
		\tilde{N}_{j1}(x,t)=
		\frac{\left(\tilde{\xi}_1^T\tilde{\zeta}_j\right)(x,t)}
		{q_j-w_1},\quad
		\tilde{N}_{j2}(x,t)=
		\frac{\left(\tilde{\xi}_3^T\tilde{\zeta}_j\right)(x,t)}
		{(q_j-w_1)^2}
		+\frac{\left(\tilde{\xi}_2^T\tilde{\zeta}_j\right)(x,t)}
		{q_j-w_1},\quad j=1,2,
	\end{equation}
	where
	\begin{subequations}
		\label{txi-zeta}
		\begin{align}
			\label{txi-zeta-a}
			&\tilde{\xi}_j(x,t)=\begin{pmatrix}
				\frac{\tilde{c}_j(x,t)}{(w_1-q_1)(w_1-q_2)}
				\\1
			\end{pmatrix}\quad j=1,3,\quad
			\tilde{\xi}_2(x,t)=\begin{pmatrix}
				\frac{\tilde{c}_1(x,t)(q_1+q_2-2w_1)}{(w_1-q_1)^2(w_1-q_2)^2}
				+\frac{\tilde{c}_2(x,t)}{(w_1-q_1)(w_1-q_2)}
				\\0
			\end{pmatrix},\\
			\label{txi-zeta-b}
			&\tilde{\zeta}_j(x,t)=\begin{pmatrix}
				\frac{(-1)^j(q_1-q_2)}{(q_j-w_1)^2}
				\tilde{f}_j(x,t)
				\\1
			\end{pmatrix}
			\quad j=1,2.
		\end{align}
	\end{subequations}
	
	Then problem \eqref{RHtM} can be solved explicitly
	by \eqref{dMtrep}, \eqref{dAjz}, and \eqref{dzjsys},
	for all values of the parameters $x,t$ such that $\det\tilde{N}(x,t)\neq 0$.
	Moreover, the large $k$ limit of the 
	$(1,2)$ and $(2,1)$ elements of $\tilde{M}$
	can be found as follows:
	\begin{equation}\label{tM12sol}
		\lim_{k\to\infty}k\tilde{M}_{12}(x,t,k)=
		\frac{\det\tilde{N}_1(x,t)}{\det\tilde{N}(x,t)},
		\quad \det\tilde{N}(x,t)\neq0,
	\end{equation}
	and
	\begin{equation}\label{tM21sol}
		\lim_{k\to\infty}k\tilde{M}_{21}(x,t,k)=
		\frac{\det\tilde{N}_2(x,t)}{\det\tilde{N}(x,t)},
		\quad \det\tilde{N}(x,t)\neq0,
	\end{equation}
	where ($\tilde{\xi}_{j,1}$ and $\tilde{\zeta}_{i,1}$ denote 
	the $(1,1)$ element of $\xi_j$ and $\zeta_i$, $j=1,2,3$, $i=1,2$, respectively)
	\begin{equation}\label{tN-1-xi-zeta}
		\tilde{N}_1(x,t)=
		\begin{pmatrix}
			\tilde{N}_{12}&\zeta_{1,1}(x,t)\\
			\tilde{N}_{22}&\zeta_{2,1}(x,t)\\
		\end{pmatrix},
	\end{equation}
	and
	\begin{equation*}
		\tilde{N}_2(x,t)=
		\left(
		\begin{array}{@{}c|c@{}}
			\tilde{N}(x,t) &
			\begin{matrix}
				1\\
				1
			\end{matrix}\\
			\hline \\ [-2ex]
			\begin{matrix}
				\tilde{\xi}_{1,1}(x,t)&\tilde{\xi}_{2,1}(x,t)
			\end{matrix}
			&0
		\end{array}
		\right).
	\end{equation*}
\end{lemma}
\begin{proof}
	We find the solution $\tilde{M}(x,t,k)$ of \eqref{dRHtM} in the following form (cf.~\eqref{Mtrep}):
	\begin{equation}
		\label{dMtrep}
		\tilde{M}(x,t,k)=\left(I+\frac{A_1(x,t)}{k-w_1}+
		\frac{A_2(x,t)}{(k-w_1)^2}\right)
		\begin{pmatrix}
			1&0\\
			0&\frac{(k-w_1)^2}{(k-q_1)(k-q_2)}
		\end{pmatrix},
	\end{equation}
	with some $2\times2$ matrices $A_j(x,t)$, $j=1,2$.
	We conclude from \eqref{dMtrep} that
	\begin{subequations}
		\label{dA-col}
	\begin{align}
		\label{dA-col-a}
		&\underset{k=w_1}{\operatorname{Res}} \tilde{M}^{(1)}(x,t,k)
		=A_1^{(1)}(x,t),\\
		\label{dA-col-b}
		&\underset{k=w_1}{\operatorname{Res}}(k-w_1)\tilde{M}^{(1)}(x,t,k)
		=A_2^{(1)}(x,t),
	\end{align}
	\end{subequations}
	and 
	\begin{subequations}
		\label{dA-col1}
	\begin{align}
		\label{dA-col1-a}
		&\tilde{M}^{(2)}(x,t,w_1)=\frac{A_2^{(2)}(x,t)}{(w_1-q_1)(w_1-q_2)},\\
		\label{dA-col1-b}
		&\partial_k\tilde{M}^{(2)}(x,t,w_1)=
		\frac{A_1^{(2)}(x,t)}{(w_1-q_1)(w_1-q_2)}
		+\frac{(q_1+q_2-2w_1)}{(w_1-q_1)^2(w_1-q_2)^2}A_2^{(2)}(x,t).
	\end{align}
	\end{subequations}
	Equations \eqref{dRHtMb}, \eqref{dA-col-b}, and \eqref{dA-col1-a}
	imply that (recall \eqref{txi-zeta-a})
	\begin{equation}\label{dA_2^1}
		A_2^{(1)}(x,t)=\tilde{\xi}_{3,1}(x,t)A_2^{(2)}(x,t),
	\end{equation}
	while from \eqref{dRHtMa}, \eqref{dA-col-a}, and \eqref{dA-col1}
	we obtain that
	\begin{equation}\label{dA_1^1}
		A_1^{(1)}(x,t)=\tilde{\xi}_{1,1}(x,t)A_1^{(2)}(x,t)
		+\tilde{\xi}_{2,1}(x,t)A_2^{(2)}(x,t).
	\end{equation}
	Combining \eqref{dA_2^1} and \eqref{dA_1^1}, we arrive at the following representations
	for $A_j(x,t)$, $j=1,2$ (cf.~\eqref{Ajz}):
	\begin{equation}\label{dAjz}
		A_1(x,t)=\left(z_1\tilde{\xi}_1^T+z_2\tilde{\xi}_2^T\right)(x,t),\quad
		A_2(x,t)=\left(z_2\tilde{\xi}_3^T\right)(x,t),
	\end{equation}
	with some $2$-vectors $z_j(x,t)$, $j=1,2$.
	
	Equation \eqref{dMtrep}, together with \eqref{dAjz}, implies that
	(cf.~\eqref{sys1})
	\begin{equation}\label{dsys1}
		I+\frac{\left(z_1\tilde{\xi}_1^T+z_2\tilde{\xi}_2^T\right)(x,t)}
		{k-w_1}
		+\frac{\left(z_2\tilde{\xi}_3^T\right)(x,t)}{(k-w_1)^2}
		=\left(\tilde{M}^{(1)}(x,t,k),
		\frac{(k-q_1)(k-q_2)}{(k-w_1)^2}\tilde{M}^{(2)}(x,t,k)\right).
	\end{equation}
	Multiplying \eqref{dsys1} by $\tilde{\zeta}_j$ on the right
	(see \eqref{txi-zeta-b}), taking $k=q_j$, $j=1,2$, and applying 
	\eqref{dRHtMc}, we obtain the following system of equations for $z_j$:
	\begin{equation*}
	\begin{split}
		\tilde{\zeta}_j(x,t)
		+\frac{\left(z_1\tilde{\xi}_1^T\tilde{\zeta}_j
			+z_2\tilde{\xi}_2^T\tilde{\zeta}_j\right)(x,t)}
		{q_j-w_1}
		+\frac{\left(z_2\tilde{\xi}_3^T\tilde{\zeta}_j\right)(x,t)}
		{(q_j-w_1)^2}=0,\quad j=1,2,
	\end{split}
	\end{equation*}
	or, in terms of $\tilde{N}$, see \eqref{tNij} (cf.~\eqref{zjsys})
	\begin{equation}\label{dzjsys}
		\tilde{N}(x,t)\begin{pmatrix}
			z_1^T\\z_2^T
		\end{pmatrix}(x,t)
		=-\begin{pmatrix}
			\tilde{\zeta}_1^T\\
			\tilde{\zeta}_2^T
		\end{pmatrix}(x,t).
	\end{equation}
	System of equations \eqref{dzjsys}, together with \eqref{dMtrep} and \eqref{dAjz}, implies that the Riemann-Hilbert problem \eqref{dRHtM} has a unique solution if $\det\tilde{N}(x,t)\neq0$.
	Finally, using that (see \eqref{dMtrep} and \eqref{dAjz};
	$z_j=\left(z_{j,1},z_{j,2}\right)^T$)
	\begin{equation*}
	\begin{split}
		&\lim_{k\to\infty}k\tilde{M}_{12}(x,t,k)=\left(A_1\right)_{12}(x,t)
		=z_{1,1}(x,t),\\
		&\lim_{k\to\infty}k\tilde{M}_{21}(x,t,k)=\left(A_1\right)_{21}(x,t)
		=\left(z_{1,2}\tilde\xi_{1,1}+z_{2,2}\tilde\xi_{2,1}\right)(x,t),
	\end{split}
	\end{equation*}
	and employing the Cramer's rule to\eqref{dzjsys}, we arrive at \eqref{tM12sol} and \eqref{tM21sol}.
\end{proof}

Having in hand Lemmas \ref{L1nm} and \ref{L2nm},
we solve the basic Riemann-Hilbert problem in the reflectionless case, thus obtaining the two-soliton solutions of the nmKdV equation satisfying
boundary conditions \eqref{bcs}.

\begin{theorem}\label{Thtws}
	Consider $b(k)=0$ and $a_j(k)$, $j=1,2$, given by \eqref{a1rc}.
	Introduce
	\begin{equation}\label{phij}
	\begin{split}
		&\phi(x,t)=2Bx+8B^3t,\\
	\end{split}
	\end{equation}
	and 
	\begin{equation}\label{s1-2}
	\begin{split}
		&s_1=\sqrt{A^2-16B^2},\quad \mbox{for }\,0<B<\frac{A}{4},\\
		&s_2=\sqrt{16B^2-A^2},\quad \mbox{for }\,B>\frac{A}{4}.
	\end{split}
	\end{equation}
	Then the nonlocal mKdV equation admits the following 
	kink-type two-soliton solutions corresponding to Cases 
	$\widetilde{\mathrm{I}}$--$\widetilde{\mathrm{III}}$
	(in the expressions below for these solutions, we consider only $(x,t)\in\R^2$ such that the denominator is nonzero; see also Remark \ref{Blup}).
	
	\begin{description}[font=\normalfont, itemindent=-\parindent]
		\item[\textit{Case} $\widetilde{\mathrm{I}}$]
		introducing 
		\begin{equation}\label{phiaj}
			\phi_j(x,t)=-2k_jx+8k_j^3t,\quad j=1,2,
		\end{equation}
		with $k_j$
		given by \eqref{zrlca}, $j=1,2$,
		the two-soliton solution reads as follows
		(see Figure \ref{fusI}; we drop the arguments of $\phi(x,t)$ and $\phi_j(x,t)$, $j=1,2$,
		for simplicity):
		\begin{equation}\label{usI}
			u(x,t)=
			\frac{A\left(s_1\cos\phi-\frac{1}{2}
				\left(
				A(\gamma_1e^{\phi_1}-\gamma_2e^{\phi_2})
				+s_1(\gamma_1e^{\phi_1}+\gamma_2e^{\phi_2})
				\right)\right)}
			{s_1-s_1(\gamma_1e^{\phi_1}+\gamma_2e^{\phi_2})\cos\phi
				-4B(\gamma_1e^{\phi_1}-\gamma_2e^{\phi_2})\sin\phi
				+\gamma_1\gamma_2s_1e^{\phi_1+\phi_2}},
		\end{equation}
		where $\gamma_j=\pm1$, $j=1,2$;
		\item[\textit{Case} $\widetilde{\mathrm{II}}$]
		introducing
		\begin{equation}\label{phi34}
			\phi_3(x,t)=-\frac{A}{2}\left(x+t\left(12B^2-A^2\right)\right),
			\quad
			\phi_4(x,t)=-\frac{s_2}{2}
			\left(x+t\left(4B^2-A^2\right)\right),
		\end{equation}
		the two-soliton solution has the following form
		(see Figure \ref{fusII}; we drop the arguments of $\phi(x,t)$ and $\phi_j(x,t)$, $j=3,4$):
		\begin{equation}\label{usII}
			u(x,t)=\frac{A\left(s_2\cos\phi+\eta_1e^{\phi_3}
				\left(A\sin\phi_4-s_2\cos\phi_4\right)\right)}
				{s_2\left(e^{2\phi_3}+1\right)+2\eta_1e^{\phi_3}
				\left(4B\sin\phi_4\sin\phi-s_2\cos\phi_4\cos\phi\right)},
		\end{equation}
		with $\eta_1=\pm 1$;
		\item[\textit{Case} $\widetilde{\mathrm{III}}$]
		introducing
		\begin{equation}\label{phi5}
			\phi_5(x,t)=-2\ell_1x+8\ell_1^3t,
		\end{equation}
		where $\ell_1$ is defined in \eqref{zrlcc}, the two-soliton solution is given by (see Figure \ref{fusIII}; we drop the arguments of $\phi(x,t)$ and $\phi_5(x,t)$)
		\begin{equation}\label{usIII}
			u(x,t)=\frac{A\left(\cos\phi-\nu_1e^{\phi_5}
				-\frac{\nu_1}{2}e^{\phi_5}\left(Ax-\frac{3}{4}A^3t\right)
				\right)}
			{e^{2\phi_5}-\nu_1e^{\phi_5}\left(2\cos\phi
				+\left(Ax-\frac{3}{4}A^3t\right)\sin\phi\right)+1},
		\end{equation}
		where $\nu_1=\pm 1$.
	\end{description}
\end{theorem}
\begin{figure}
	\centering{\includegraphics[scale=0.6]
		{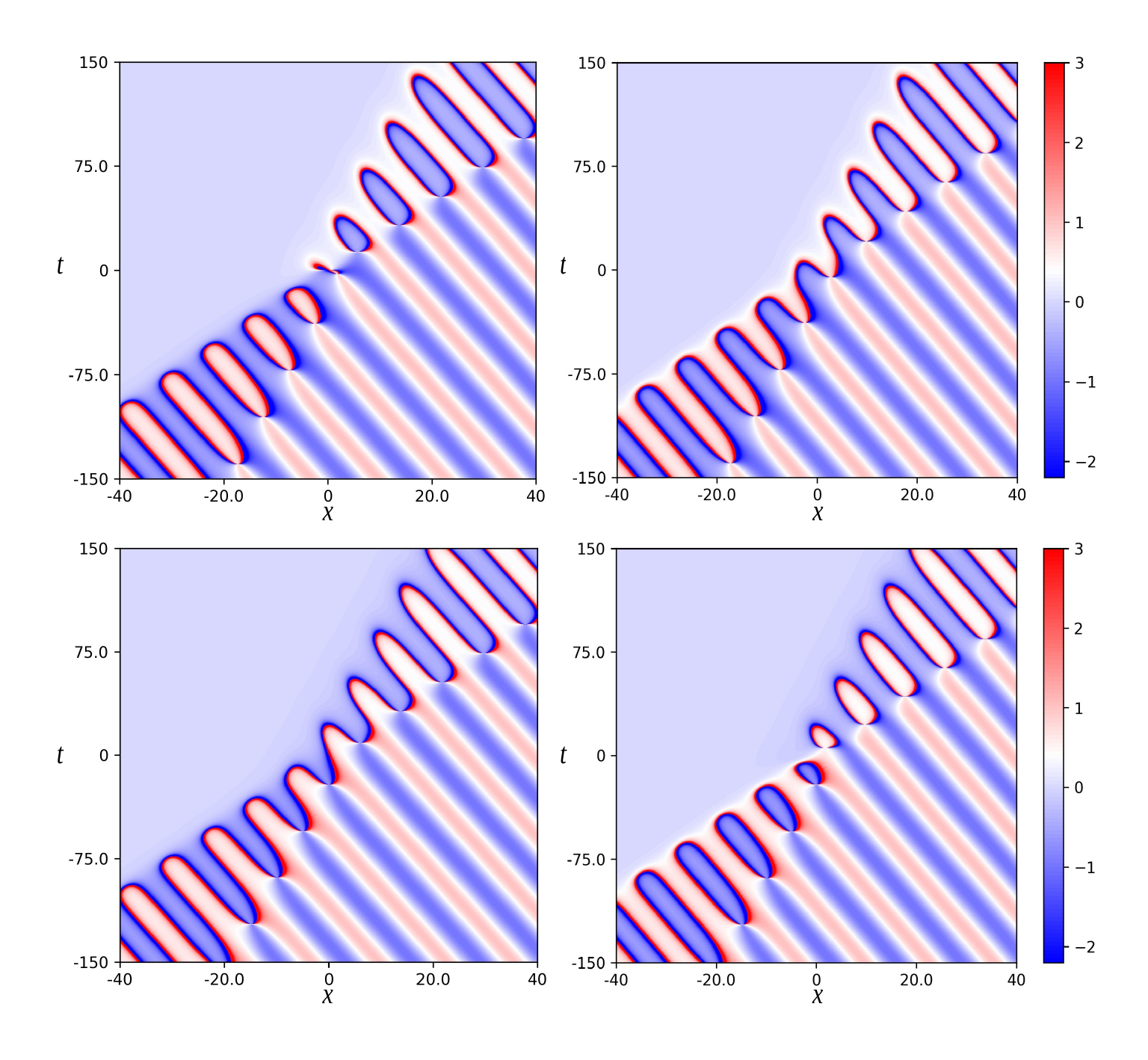}}
	\caption{The graphs of the two-soliton solutions, which correspond to
		two pure imaginary simple zeros at $k=\I k_j$, $j=1,2$, of $a_1(k)$,
		see \eqref{zrlca}.
		These graphs are plotted from the exact formula 
		\eqref{usI} with $A=1$, $B=0.243$,
		and $(\gamma_1,\gamma_2)=(1,1)$ (top left),
		$(\gamma_1,\gamma_2)=(1,-1)$ (top right),
		$(\gamma_1,\gamma_2)=(-1,1)$ (bottom left),
		and $(\gamma_1,\gamma_2)=(-1,-1)$ (bottom right).}
	\label{fusI}
\end{figure}
\begin{figure}
	\centering{\includegraphics[scale=0.6]
		{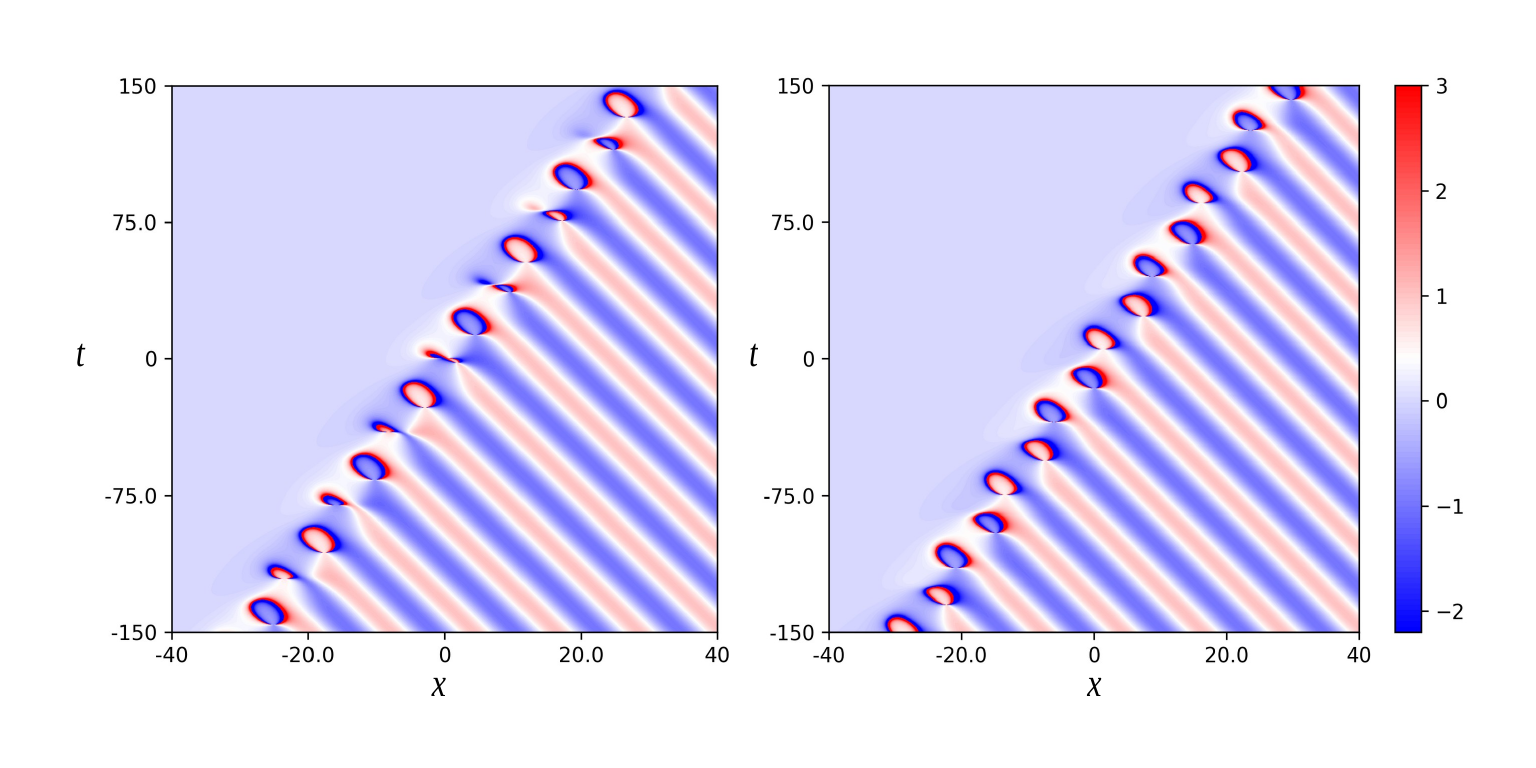}}
	\caption{The graphs of the two-soliton solutions, which correspond to
		two simple zeros at $k=p_1$ and $k=-\bar{p}_1$ of $a_1(k)$, see \eqref{zrlcb}.
		These graphs are plotted from the exact formula 
		\eqref{usII} with $A=1$, $B=0.26$,
		and $\eta_1=1$ (left),
		and $\eta_1=-1$ (right).}
	\label{fusII}
\end{figure}
\begin{figure}
	\centering{\includegraphics[scale=0.6]
		{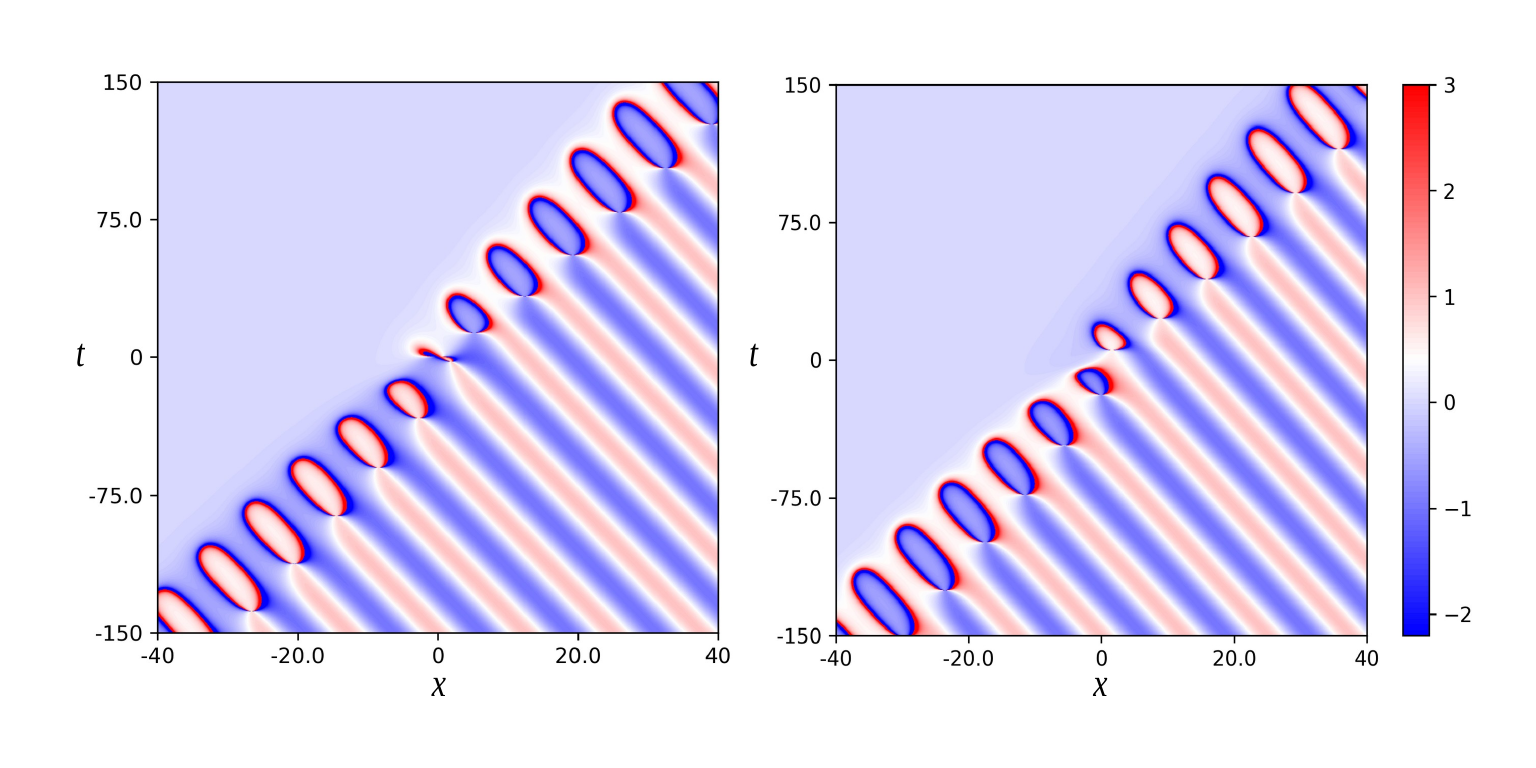}}
	\caption{The graphs of the two-soliton solutions, which correspond to
		one double zero at $k=\I\ell_1$ of $a_1(k)$, see \eqref{zrlcc}.
		These graphs are plotted from the exact formula 
		\eqref{usIII} with $A=1$, $B=\frac{1}{4}$,
		and $\nu_1=1$ (left),
		and $\nu_1=-1$ (right).}
	\label{fusIII}
\end{figure}
\begin{proof}
	\textbf{Case $\widetilde{\mathrm{\mathbf{I}}}$.}
	Here $M(x,t,k)$ solves the Riemann-Hilbert problem \eqref{RHtM} in Lemma \ref{L1nm}
	with (here $k_j$, $j=1,2$, are given by \eqref{zrlca})
	\begin{equation}\label{wjI}
		(w_1,w_2)=(\I k_1,\I k_2),\quad (q_1,q_2)=(B,-B),
	\end{equation}
	and with $c_j(x,t)$, $f_j(x,t)$, $j=1,2$, having the following form
	(see \eqref{rskj}, \eqref{a1rc}, \eqref{res+-B}; recall also \eqref{phij}, \eqref{phiaj}):
	\begin{equation}\label{cjI}
	\begin{split}
		&c_j(x,t)=(-1)^j\I\frac{\gamma_j(k_j^2+B^2)}{k_2-k_1}
			e^{\phi_j(x,t)},
			\quad j=1,2,\\
		&f_j(x,t)=-\I\frac{A}{4}
		e^{(-1)^j\I\phi(x,t)}
		\quad j=1,2.
	\end{split}
	\end{equation}
	Combining \eqref{xi-zeta}, \eqref{wjI}, and \eqref{cjI}, we obtain the following expressions for $\xi_j$ and $\zeta_j$, $j=1,2$:
	\begin{equation}\label{xi-zeta-I}
		\xi_j(x,t)=\begin{pmatrix}
		\gamma_je^{\phi_j(x,t)}
		\\1
		\end{pmatrix},\quad
		\zeta_j(x,t)=\begin{pmatrix}
		-e^{(-1)^j\I\phi(x,t)}
		\\1
		\end{pmatrix}
		\quad j=1,2,
	\end{equation}
	where we have used that $(B+\I k_1)(B+\I k_2)=\I\frac{AB}{2}$,
	see \eqref{zrlca}.
	
	Combining \eqref{Nij}, \eqref{wjI}, and \eqref{xi-zeta-I}, we obtain the following equations for the elements of the matrix $N(x,t)$:
	\begin{equation}\label{INij}
		N_{ij}(x,t)=\frac{(-1)^{i+1}}{B+(-1)^i\I k_j}
		\left(1-\gamma_je^{\phi_j(x,t)+(-1)^i\I\phi(x,t)}\right),\quad
		i,j=1,2.
	\end{equation}
	Equations \eqref{N-1-xi-zeta}, \eqref{xi-zeta-I}, and \eqref{INij} imply
	that (here and below we drop the arguments of
	$\phi(x,t)$ and $\phi_j(x,t)$, $j=1,2$)
	\begin{equation}\label{N1I}
	\begin{split}
		\det N_1(x,t)&=2\mathrm{Re}\,\left(\frac{\I(k_1-k_2)}
		{(B-\I k_1)(B-\I k_2)}e^{\I\phi}\right)
		-\frac{2\gamma_1B}{B^2+k_1^2}e^{\phi_1}
		+\frac{2\gamma_2B}{B^2+k_2^2}e^{\phi_2}\\
		&=\frac{2s_1}{AB}\cos\phi
		-\frac{16\gamma_1B}{A^2-As_1}e^{\phi_1}
		+\frac{16\gamma_2B}{A^2+As_1}e^{\phi_2},
	\end{split}
	\end{equation}
	where we have used that $B^2+k_j^2=\frac{A}{8}(A+(-1)^js_1)$, $j=1,2$,
	and $k_2-k_1=s_1/2$.
	
	Taking into account that
	$$
	\frac{1}{(B+\I k_1)(B-\I k_2)}=\frac{8}{A^2}+\I\frac{2s_1}{A^2B},
	$$
	and using \eqref{INij}, we obtain that
	\begin{equation}\label{NI}
	\begin{split}
		\det N(x,t)&=2\I\mathrm{Im}\left[
		\frac{1-\gamma_1e^{\phi_1+\I\phi}-\gamma_2e^{\phi_2-\I\phi}
			+\gamma_1\gamma_2e^{\phi_1+\phi_2}}
			{(B+\I k_1)(B-\I k_2)}\right]\\
			&=\frac{4\I s_1}{A^2B}\left(1
			-\left(\gamma_1e^{\phi_1}+\gamma_2e^{\phi_2}\right)\cos\phi
			-\frac{4B}{s_1}
			\left(\gamma_1e^{\phi_1}-\gamma_2e^{\phi_2}\right)\sin\phi
			+\gamma_1\gamma_2e^{\phi_1+\phi_2}
			\right).
	\end{split}
	\end{equation}
	Combining \eqref{msol}, \eqref{M12sol}, \eqref{N1I}, and \eqref{NI}, we arrive at \eqref{usI}.
	\medskip
	
	\textbf{Case $\widetilde{\mathrm{\mathbf{II}}}$.}
	In this case, we use Lemma \ref{L1nm} with
	\begin{equation}
		\label{wjII}
		(w_1,w_2)=(p_1,-\bar{p}_1),\quad (q_1,q_2)=(B,-B),
	\end{equation}
	where $p_1$ is given by \eqref{zrlcb},
	and with $c_j(x,t)$, $f_j(x,t)$, $j=1,2$, having the following form
	(recall \eqref{rsp1}, \eqref{a1rc}, \eqref{res+-B}, and \eqref{phij}):
	\begin{equation}
		\label{cjII}
		\begin{split}
			&c_1(x,t)=\eta_1\frac{p_1^2-B^2}{2\mathrm{Re}\,p_1}
			e^{2\I p_1x+8\I p_1^3t},\\
			&c_2(x,t)=\eta_1\frac{B^2-\bar{p}_1^2}{2\mathrm{Re}\,p_1}
			e^{-2\I\bar{p}_1x-8\I\bar{p}_1^3t},\\
			&f_j(x,t)=-\I\frac{A}{4}
			e^{(-1)^j\I\phi(x,t)}
			\quad j=1,2.
		\end{split}
	\end{equation}
	Using that $(B-p_1)(B+\bar{p}_1)=-\I\frac{AB}{2}$ and 
	(see \eqref{zrlcb} and \eqref{phi34})
	\begin{equation*}
		2\I p_1x+8\I p_1^3t=\phi_3+\I\phi_4,
	\end{equation*} 
	we conclude from \eqref{wjII} and \eqref{cjII} that
	$\xi_j$ and $\zeta_j$, $j=1,2$, see \eqref{xi-zeta}, 
	read as follows:
	\begin{equation}
		\label{xi-zeta-II}
		\xi_1(x,t)=\begin{pmatrix}
			\eta_1e^{\phi_3+\I\phi_4}
			\\1
		\end{pmatrix},\quad
		\xi_2(x,t)=\begin{pmatrix}
			\eta_1e^{\phi_3-\I\phi_4}
			\\1
		\end{pmatrix},\quad
		\zeta_j(x,t)=\begin{pmatrix}
			-e^{(-1)^j\I\phi(x,t)}
			\\1
		\end{pmatrix},
		\quad j=1,2.
	\end{equation}
	Then, equations \eqref{Nij}, \eqref{wjII}, and \eqref{xi-zeta-II} yield
	the following expressions for the elements of the matrix $N(x,t)$
	(here and below we drop the arguments $x,t$ for simplicity):
	\begin{equation}
	\begin{split}
		\label{IINij}
		&N_{11}=\frac{1}{B-p_1}
		\left(1-\eta_1e^{\phi_3+\I(\phi_4-\phi)}\right),\quad
		N_{21}=-\frac{1}{B+p_1}
		\left(1-\eta_1e^{\phi_3+\I(\phi_4+\phi)}\right),\\
		&N_{12}=\frac{1}{B+\bar{p}_1}
		\left(1-\eta_1e^{\phi_3-\I(\phi_4+\phi)}\right),\quad
		N_{22}=\frac{1}{\bar{p}_1-B}
		\left(1-\eta_1e^{\phi_3-\I(\phi_4-\phi)}\right).
	\end{split}
	\end{equation}
	
	Combining \eqref{N-1-xi-zeta}, \eqref{xi-zeta-II}, and \eqref{IINij}, we conclude that
	\begin{equation}\label{N1II}
		\det N_1=2\I\,\mathrm{Im}\left(
		\frac{1}{B+p_1}\left(e^{-\I\phi}-\eta_1e^{\phi_3+\I\phi_4}\right)
		+\frac{1}{B-p_1}\left(e^{\I\phi}-\eta_1e^{\phi_3+\I\phi_4}\right)
		\right),
	\end{equation}
	and
	\begin{equation}\label{NII}
		\det N=\frac{1-2\eta_1 e^{\phi_3}\cos(\phi_4+\phi)
				+e^{2\phi_3}}{(B+p_1)(B+\bar{p}_1)}
			-\frac{1-2\eta_1 e^{\phi_3}\cos(\phi_4-\phi)
				+e^{2\phi_3}}{(B-p_1)(B-\bar{p}_1)}.
	\end{equation}
	Applying \eqref{msol} and \eqref{M12sol} together with \eqref{N1II} and \eqref{NII}, we obtain that
	\begin{subequations}
		\label{usIIb}
	\begin{equation}
		u(x,t)=\frac{F_1(x,t)}{F_2(x,t)},
	\end{equation}
	where (the arguments $x,t$ are dropped here)
	\begin{equation}
		\begin{split}
			&F_1=-4\mathrm{Im}\left[(B+\bar{p}_1)(B-\bar{p}_1)
			\left((B-p_1)
			\left(e^{-\I\phi}-\eta_1e^{\phi_3+\I\phi_4}\right)
			+(B+p_1)
			\left(e^{\I\phi}-\eta_1e^{\phi_3+\I\phi_4}\right)\right)\right],\\
			&F_2=(B-p_1)(B-\bar{p}_1)
			\left(1-2\eta_1 e^{\phi_3}\cos(\phi_4+\phi)+e^{2\phi_3}\right)\\
			&\qquad\,\,-(B+p_1)(B+\bar{p}_1)
			\left(1-2\eta_1 e^{\phi_3}\cos(\phi_4-\phi)+e^{2\phi_3}\right).			
		\end{split}
	\end{equation}
	\end{subequations}
	Taking into account that (recall \eqref{zrlcb} and \eqref{s1-2})
	\begin{equation*}
	\begin{split}
		&(B+\bar{p}_1)(B-\bar{p}_1)=\frac{A}{8}(A-\I s_2),\\
		&(B+p_1)(B+\bar{p}_1)=2B^2-\frac{Bs_2}{2},\\
		&(B-p_1)(B-\bar{p}_1)=2B^2+\frac{Bs_2}{2},
	\end{split}
	\end{equation*}
	we have \eqref{usII} from \eqref{usIIb} by direct computations.
	\medskip
	
	\textbf{Case $\widetilde{\mathrm{\mathbf{III}}}$.}
	Here $M(x,t,k)$ satisfies the Riemann-Hilbert problem \eqref{dRHtM}
	discussed in Lemma \ref{L2nm}, where (recall \eqref{zrlcc})
	\begin{equation}
		\label{w1III}
		w_1=\I\ell_1=\I\frac{A}{4},\quad (q_1,q_2)=(B,-B),
	\end{equation}
	and with $\tilde{c}_j(x,t)$, $j=1,2,3$, $\tilde{f}_j(x,t)$, $j=1,2$, having the following form
	(recall \eqref{rsl1}, \eqref{a1rc}, \eqref{res+-B}, \eqref{phij}, and
	\eqref{phi5}; 
	notice also that $B=\frac{A}{4}$, see Corollary \ref{rlc}):
	\begin{equation}
	\label{cjIII}
	\begin{split}
		&\tilde{c}_j(x,t)=-\nu_1\frac{A^2}{8}
		e^{\phi_5(x,t)},\quad j=1,3,\\
		&\tilde{c}_2(x,t)=-\I\nu_1\frac{A^2}{4}
		\left(x-\frac{3}{4}A^2t-\frac{2}{A}\right)
		e^{\phi_5(x,t)},\\
		&\tilde{f}_j(x,t)=-\I\frac{A}{4}
		e^{(-1)^j\I\phi(x,t)}
		\quad j=1,2.
	\end{split}
	\end{equation}
	In \eqref{cjIII} we have used that
	\begin{equation*}
		a_1^{\prime\prime}(\I\ell_1)=-\frac{16}{A^2},\quad
		\mbox{and}\quad
		a_1^{\prime\prime\prime}(\I\ell_1)=-\I\frac{192}{A^3}.
	\end{equation*}
	Combining \eqref{w1III}, \eqref{txi-zeta}, and \eqref{cjIII}, we conclude that $\tilde{\xi}_j$, $j=1,2,3$, and $\tilde{\zeta}_j$, $j=1,2$, have the following form:
	\begin{equation}
		\label{txi-zeta-III}
	\begin{split}
		&\tilde{\xi}_j(x,t)=\begin{pmatrix}
			\nu_1e^{\phi_5(x,t)}
			\\1
		\end{pmatrix}\quad j=1,3,\quad
		\tilde{\xi}_2(x,t)=\begin{pmatrix}
			2\I\nu_1\left(x-\frac{3}{4}A^2t\right)e^{\phi_5(x,t)}
			\\0
		\end{pmatrix},\\
		&\tilde{\zeta}_j(x,t)=\begin{pmatrix}
			-e^{(-1)^j\I\phi(x,t)}
			\\1
		\end{pmatrix}
		\quad j=1,2.
	\end{split}
	\end{equation}
	From Equations \eqref{tNij}, \eqref{w1III}, and \eqref{txi-zeta-III}
	we obtain the following expressions for the elements of the matrix
	$\tilde{N}(x,t)$
	(here and below the arguments $x,t$ are dropped for simplicity):
	\begin{subequations}\label{IIINij}
	\begin{equation}
		\tilde{N}_{11}=\frac{4}{A-\I A}
		\left(1-\nu_1e^{\phi_5-\I\phi}\right),\quad
		\tilde{N}_{21}=-\frac{4}{A+\I A}
		\left(1-\nu_1e^{\phi_5+\I\phi}\right),
	\end{equation}
	and
	\begin{equation}
	\begin{split}
		&\tilde{N}_{12}=-\I\frac{8}{A^2}
		\left(\nu_1e^{\phi_5-\I\phi}
		+\frac{\nu_1(1+\I)}{2}
		\left(Ax-\frac{3}{4}A^3t\right)e^{\phi_5-\I\phi}-1\right),\\
		&\tilde{N}_{22}=\I\frac{8}{A^2}
		\left(\nu_1e^{\phi_5+\I\phi}
		+\frac{\nu_1(1-\I)}{2}
		\left(Ax-\frac{3}{4}A^3t\right)e^{\phi_5+\I\phi}-1\right).
	\end{split}
	\end{equation}
	\end{subequations}
	Equations \eqref{tN-1-xi-zeta}, \eqref{txi-zeta-III}, and 
	\eqref{IIINij} imply that
	\begin{equation}\label{N1III}
		\det\tilde{N}_1=\I\frac{16}{A^2}
		\left(\nu_1e^{\phi_5}
		+\frac{\nu_1}{2}
		\left(Ax-\frac{3}{4}A^3t\right)e^{\phi_5}
		-\cos\phi
		\right),
	\end{equation}
	as well as
	\begin{equation}\label{NIII}
		\det\tilde{N}=\frac{32}{A^3}\left(
		e^{2\phi_5}-2\nu_1e^{\phi_5}\cos\phi
		-\nu_1\left(Ax-\frac{3}{4}A^3t\right)e^{\phi_5}\sin\phi+1
		\right).
	\end{equation}
	Finally, combining \eqref{tM12sol}, \eqref{N1III} and \eqref{NIII}, we arrive at \eqref{usIII}.
\end{proof}
\begin{remark}[Blow-up points]\label{Blup}
	Two-soliton solutions \eqref{usI}, \eqref{usII}, and \eqref{usIII} blow-up in the uniform norm at zeros of the corresponding denominators.
	These points form the curves in the $(x,t)$-plane, which can be inferred from Figures \ref{fusI}, \ref{fusII}, and \ref{fusIII} for, respectively, the solutions \eqref{usI}, \eqref{usII}, and \eqref{usIII} with different values of the norming constants $\gamma_j$, $j=1,2$, $\eta_1$, and $\nu_1$.
\end{remark}

\begin{remark}[Solutions of the nmKdV equation]
	Solitons with step-like oscillating boundary conditions \eqref{bcs}
	given in Theorem \ref{Thtws} blow-up in finite time
	for any values of the norming constants $\gamma_j$, $j=1,2$, $\eta_1$, and $\nu_1$.
	Observe that soliton solutions of the nmKdV equation with vanishing background
	can be regular for all finite $x$ and $t$
	for certain values of the spectral parameters
	\cite[Section 3]{JZ17}, \cite[Section 5]{ZTYZ24}
	(cf.~kink-type solution \cite[Equation (68)]{XF23}
	satisfying \eqref{bcs} with $B=0$).
	In this regard, we notice that the work \cite{LF24}
	discusses the global existence of solutions
	of the nmKdV equation with small, in a sense, initial data
	by employing the Riemann-Hilbert approach.
	The well-posedness questions of the related 
	nonlocal NLS equation \eqref{NNLS},
	which also admits singular solitons,
	were treated in \cite{CLW23, G17, RS23} by various methods.
\end{remark}

We close this section by giving the large-time asymptotic formulas for the soliton solutions presented in Theorem \ref{Thtws}.
\begin{corollary}
	Two-soliton solutions given in
	Cases $\widetilde{\mathrm{I}}$, $\widetilde{\mathrm{II}}$,
	and $\widetilde{\mathrm{III}}$ of Theorem \ref{Thtws}
	have the following asymptotic behavior as $t\to\infty$ along the rays
	$\frac{x}{t}=const$ (the bounds below are uniform in the corresponding regions away from any neighborhoods of possible zeros of the denominators):
	\begin{description}[font=\normalfont, itemindent=-\parindent]
		\item[\textit{Case} $\widetilde{\mathrm{I}}$]
		asymptotics of the solution \eqref{usI} reads as follows:
		\begin{enumerate}[label=(\roman*)]
			\item decaying region, $x<4k_1^2t$:
			\begin{equation*}
				u(x,t)=O(e^{-ct}),\quad c=c(x/t)>0,\quad t\to\infty;
			\end{equation*}
			
			\item transition region, $x=4k_1^2t+x^\prime$
			with any fixed $x^\prime\in\R$:
			\begin{equation}\label{tr1}
				u(x,t)=\frac{\frac{A}{2}\left(A-s_1\right)}
				{4B\sin\phi-s_1\cos\phi+\gamma_1s_1e^{-2k_1x^\prime}}
				+O(e^{-ct}),
				\quad c>0,\quad t\to\infty.
			\end{equation}
			Notice that the asymptotic behavior does not depend on the norming constant $\gamma_2$, cf.~Figure \ref{fusI};
			
			\item large-amplitude oscillations region,
			$4k_1^2t<x<4k_2^2t$:
			\begin{equation*}
				u(x,t)=\frac{\frac{A}{2}\left(A-s_1\right)}
				{4B\sin\phi-s_1\cos\phi}
				+O(e^{-ct}),
				\quad c=c(x/t)>0,\quad t\to\infty.
			\end{equation*}
			Here the asymptotics does not depend on the values of the both norming constants $\gamma_1$ and $\gamma_2$;
			
			\item transition region, $x=4k_2^2t+x^\prime$
			with any fixed $x^\prime\in\R$:
			\begin{equation*}
				u(x,t)=\frac{\frac{A}{2}\left(2s_1
					e^{2k_2x^\prime}\cos\phi+\gamma_2A
					-\gamma_2s_1\right)}
				{s_1e^{2k_2x^\prime}+4B\gamma_2\sin\phi-s_1\gamma_2\cos\phi}
				+O(e^{-ct}),
				\quad c>0,\quad t\to\infty.
			\end{equation*}
			Asymptotics in this transition region does not depend on $\gamma_1$, cf.~\eqref{tr1} and Figure \ref{fusI};
			
			\item periodic region, $x>4k_2^2t$:
			\begin{equation*}
				u(x,t)=A\cos\phi+O(e^{-ct}),
				\quad c=c(x/t)>0,\quad t\to\infty;
			\end{equation*}
		\end{enumerate}
		
		\item[\textit{Case} $\widetilde{\mathrm{II}}$]
		asymptotics of the solution \eqref{usII} reads as follows:
		\begin{enumerate}[label=(\roman*)]
			\item decaying region, $x<(A^2-12B^2)t$:
			\begin{equation*}
				u(x,t)=O(e^{-ct}),\quad c=c(x/t)>0,\quad t\to\infty;
			\end{equation*}
			
			\item transition region, $x=(A^2-12B^2)t+x^\prime$
			with any fixed $x^\prime\in\R$:
			\begin{equation*}
				u(x,t)=\frac{A\left(s_2\cos\phi+\eta_1e^{-\frac{A}{2}x^\prime}
					\left(A\sin\tilde{\phi}_4-s_2\cos\tilde{\phi_4}
					\right)\right)}
				{s_2\left(e^{-Ax^\prime}+1\right)
					+2\eta_1e^{-\frac{A}{2}x^\prime}
					\left(4B\sin\tilde{\phi_4}\sin\phi
					-s_2\cos\tilde{\phi}_4\cos\phi\right)}
				+O(e^{-ct}),
				\quad c>0,\quad t\to\infty,
			\end{equation*}
			with $\tilde{\phi}_4=\frac{s_2}{2}(8B^2t-x^\prime)$, cf.~$\phi_4$ in \eqref{phi34};
			
			\item periodic region, $x>(A^2-12B^2)t$:
			\begin{equation*}
				u(x,t)=A\cos\phi+O(e^{-ct}),
				\quad c=c(x/t)>0,\quad t\to\infty;
			\end{equation*}
		\end{enumerate}
		
		\item[\textit{Case} $\widetilde{\mathrm{III}}$]
		asymptotics of the solution \eqref{usIII} reads as follows (recall that here $\ell_1=\frac{A}{4}$):
		\begin{enumerate}[label=(\roman*)]
			\item decaying region, $x<\frac{A^2}{4}t$:
			\begin{equation*}
				u(x,t)=O(e^{-ct}),\quad c=c(x/t)>0,\quad t\to\infty;
			\end{equation*}
			
			\item transition region, $x=\frac{A^2}{4}t+x^\prime$ with any fixed $x^\prime\in\R$:
			\begin{equation*}
				u(x,t)=\frac{A\left(e^{2\ell_1x^\prime}\cos\phi-\nu_1
					-\frac{\nu_1}{2}
					\left(Ax^\prime-\frac{1}{2}A^3t\right)
					\right)}
				{e^{-2\ell_1x^\prime}-\nu_1\left(2\cos\phi
					+\left(Ax^\prime-\frac{1}{2}A^3t\right)\sin\phi\right)
					+e^{2\ell_1x^\prime}}
					+O(e^{-ct}),
					\quad c>0,\quad t\to\infty;
			\end{equation*}
			
			\item periodic region, $x>\frac{A^2}{4}t$:
			\begin{equation*}
				u(x,t)=A\cos\phi+O(e^{-ct}),
				\quad c=c(x/t)>0,\quad t\to\infty.
			\end{equation*}
		\end{enumerate}	
	\end{description}
	Similarly, one can obtain asymptotic formulas of the soliton solutions 
	\eqref{usI}, \eqref{usII}, and \eqref{usIII} as $t\to-\infty$;
	we omit the precise formulas here for brevity.
\end{corollary}
\medskip

\textbf{Data availability statement.}
Data availability is not applicable to this article as no new data were created or analysed in this study.

\textbf{Conflict of interest.}
The author declares no conflict of interest.


\end{document}